\documentclass[journal,twoside,web]{ieeecolor}
\usepackage{generic}
\usepackage{cite}
\usepackage{amsmath,amssymb,amsfonts}
\usepackage{graphicx}
\usepackage{textcomp}

\usepackage{float}
\usepackage{tabularx}
\usepackage{algorithm}
\usepackage{algorithmicx}
\usepackage{algpseudocode}
\usepackage{float}
\usepackage{subfig}
\usepackage{caption}
\usepackage{booktabs}
\usepackage{enumerate}
\usepackage{todonotes}
\usepackage{makecell}

\usepackage{url}
\usepackage{color,hyperref}

\newcommand{\order}[1]{\mathcal{O}\left(#1\right)}
\newcommand{\prt}[1]{\left(#1\right)}
\newcommand{\brk}[1]{\left[#1\right]}
\newcommand{\crk}[1]{\left\{#1\right\}}

\newcommand{\orderi}[1]{\mathcal{O}(#1)}
\newcommand{\prti}[1]{(#1)}
\newcommand{\brki}[1]{[#1]}
\newcommand{\crki}[1]{\{#1\}}
\newcommand{\condEi}[2]{\E\brki{#1|#2}}
\newcommand{\normi}[1]{\Vert #1\Vert}
\newcommand{\inproi}[1]{\langle #1 \rangle}

\newcommand{\inpro}[1]{\left\langle #1 \right\rangle}
\newcommand{\condE}[2]{\E\brk{#1\middle|#2}}
\newcommand{\norm}[1]{\left\Vert #1 \right\Vert}
\newcommand{\bs}[1]{\boldsymbol{#1}}
\newcommand{\R}{\mathbb{R}}
\newcommand{\E}{\mathbb{E}}
\newcommand{\x}{\mathbf{x}}
\newcommand{\y}{\mathbf{y}}
\newcommand{\s}{\mathbf{s}}

\newcommand{\g}{\mathbf{g}}
\newcommand{\1}{\mathbf{1}}
\newcommand{\T}{\intercal}
\newcommand{\sumn}{\sum_{i=1}^n}
\newcommand{\compress}[1]{\mathcal{C}\prt{#1}}
\newcommand{\tW}{\tilde{W}}

\newcommand{\tV}{\tilde{V}}
\newcommand{\tlambda}{\tilde{\lambda}}
\newcommand{\Span}[1]{\mathrm{Span}\prt{#1}}

\renewcommand{\r}{\mathbf{r}}
\newcommand{\diag}{\mathrm{diag}}
\newcommand{\tLambda}{\tilde{\Lambda}}

\newcommand{\cG}{\mathcal{G}}
\newcommand{\cF}{\mathcal{F}}
\newcommand{\cL}{\mathcal{L}}
\newcommand{\cN}{\mathcal{N}}
\newcommand{\0}{\mathbf{0}}
\newcommand{\ch}{\check{h}}
\newcommand{\cA}{\mathcal{A}}

\newcommand{\cC}{\mathcal{C}}
\newcommand{\bB}{\mathbb{B}}
\newcommand{\bU}{\mathbb{U}}

\newcommand{\finf}{f^{\text{inf}}}
\newcommand{\ul}[1]{{#1}}
\newcommand{\cO}{\mathcal{O}}

\newtheorem{theorem}{Theorem}
\newtheorem{lemma}{Lemma}
\newtheorem{assumption}{Assumption}

\newtheorem{remark}{Remark}
\newtheorem{corollary}{Corollary}

\definecolor{cuhkpl}{RGB}{152,24,147} 	

\def\BibTeX{{\rm B\kern-.05em{\sc i\kern-.025em b}\kern-.08em
    T\kern-.1667em\lower.7ex\hbox{E}\kern-.125emX}}
\begin{document}
\title{CEDAS: A Compressed Decentralized Stochastic Gradient Method with Improved Convergence}
\author{Kun Huang, and 
Shi Pu, \IEEEmembership{Senior Member, IEEE}
\thanks{This work is partially supported by National Natural Science Foundation of China (NSFC) (Grant No. 62003287 and No. 62373316), Guangdong Talent Program (Grant No. 2021QN02X216), and Shenzhen Science and Technology Program (Grant No. RCYX20210609103229031).}
\thanks{K. Huang and S. Pu are with the School
of Data Science, The Chinese
University of Hong Kong, Shenzhen (CUHK-Shenzhen), China.
{\tt\small (emails: kunhuang@link.cuhk.edu.cn, pushi@cuhk.edu.cn)}
}
}
\maketitle

\begin{abstract}
    In this paper, we consider solving the distributed optimization problem over a multi-agent network under the communication restricted setting. We study a compressed decentralized stochastic gradient method, termed ``compressed exact diffusion with adaptive stepsizes (CEDAS)", and show the method asymptotically achieves comparable convergence rate as centralized { stochastic gradient descent (SGD)} for both smooth strongly convex objective functions and smooth nonconvex objective functions under unbiased compression operators. 
	In particular, to our knowledge, CEDAS enjoys so far the shortest transient time (with respect to the graph specifics) 
	for achieving the convergence rate of centralized SGD, which behaves as $\mathcal{O}(n{C^3}/(1-\lambda_2)^{2})$ under smooth strongly convex objective functions, and $\mathcal{O}(n^3{C^6}/(1-\lambda_2)^4)$ under smooth nonconvex objective functions, where $(1-\lambda_2)$ denotes the spectral gap of the mixing matrix, and $C>0$ is the compression-related parameter. 
 {In particular, CEDAS exhibits the shortest transient times  when
$C < \orderi{1/(1 - \lambda_2)^2}$, which is common in practice.}
 Numerical experiments further demonstrate the effectiveness of the proposed algorithm.
\end{abstract}

\begin{IEEEkeywords}
    convex optimization, nonconvex optimization, distributed optimization, stochastic gradient methods, compressed gradient methods
\end{IEEEkeywords}

\section{Introduction}

\IEEEPARstart{I}{n} this paper, we consider solving the following distributed optimization problem by a group of agents $[n] := \crk{1,2,\cdots, n}$ connected over a network:
\begin{equation}
	\label{eq:P}
	\min_{x\in\R^p} f(x) = \frac{1}{n}\sumn f_i(x),
\end{equation}
where $f_i$ is the local cost function known by agent $i$ only. Problem \eqref{eq:P} appears naturally in many machine learning and signal processing problems, where each $f_i$ represents an expected or empirical risk function corresponding to the local dataset of agent $i$. Solving problem \eqref{eq:P} in a decentralized manner over a multi-agent network has gained great interests in recent years, partly because
decentralization helps relieve the high latency at the central server in a centralized computing architecture \cite{nedic2018network}.
However, decentralized optimization algorithms may still suffer from the communication bottleneck under the huge size of modern machine learning models and/or limited bandwidth. For example, the gradients used for training LSTM models can be as large as $110.43$ MB \cite{koloskova2019decentralized}. 

Communication compression techniques, such as quantization \cite{alistarh2017qsgd,bernstein2018signsgd,seide20141} and sparsification \cite{wangni2018gradient,stich2018sparsified}, are one of the most effective means for reducing the communication cost in distributed computation. The performance of several commonly used communication compression methods has been well studied under the centralized master-worker architecture, and the algorithms have been shown to enjoy comparable convergence rates with their non-compressed counterparts when equipped with various techniques such as error compensation \cite{seide20141,richtarik2021ef21,alistarh2018convergence,stich2018sparsified} and gradient difference compression \cite{mishchenko2019distributed}. 
It is thus natural to integrate decentralization and communication compression to benefit from both techniques. The work in \cite{tang2018communication} combined distributed stochastic gradient descent (DSGD) with quantization for the first time, followed by \cite{chocosgd,koloskova2019decentralized} which extended the compressed DSGD method to fit a richer set of compression operators. The paper \cite{liu2021linear} equipped NIDS \cite{nids} with communication compression and demonstrated its linear convergence rate when the objective functions are smooth and strongly convex. More recently, the papers \cite{liao2021compressed,song2021compressed} considered a class of linearly convergent decentralized gradient tracking method with communication compression that applies to general directed graphs; see also \cite{zhang2021innovation,yi2022communication} for more related works.

It is worth noting that introducing decentralization may slow down the algorithmic convergence as the network size grows when compared to the centralized master-worker architecture \cite{pu2020asymptotic}. Therefore, an important question arises when considering decentralized optimization with communication compression, that is,
\textit{can a compressed decentralized (stochastic) gradient method achieve similar convergence rate compared to its centralized counterpart?} 
Such a question has been answered positively in \cite{tang2018communication} first with several works following \cite{chocosgd,koloskova2019decentralized,singh2022sparq,singh2021squarm,vogels2020practical}. However, these methods are all variants of DSGD which suffer from the data heterogeneity \cite{nedic2009distributed,pu2021sharp,lian2017can}. As a result, significant amount of transient times are often required for achieving comparable convergence rate with centralized SGD.

In this paper, we consider a novel method for solving Problem \eqref{eq:P}, termed ``compressed exact diffusion algorithm with adaptive stepsizes (CEDAS)", which is adapted from EXTRA \cite{shi2015extra}, exact diffusion \cite{yuan2018exact} and the LEAD algorithm \cite{liu2021linear}. We analyze CEDAS under an unbiased compression operator for both smooth strongly convex objective functions and smooth nonconvex objective functions  { without any data heterogeneity assumptions}. 
We are able to show that under both scenarios, the CEDAS method enjoys linear speedup similar to a centralized SGD method without communication compression.
In particular, the performance of CEDAS outperforms the state-of-the-art algorithms in terms of the transient time 
to achieve the convergence rate of centralized SGD. 
These results suggest that CEDAS benefits from both communication compression and decentralization better than the existing algorithms.

\subsection{Related Works}

There is a vast literature on solving Problem \eqref{eq:P} under the communication restricted settings. For example, under the centralized master-worker architecture, the works in \cite{seide20141,strom2015scalable,alistarh2017qsgd,de2017understanding,wangni2018gradient,bernstein2018signsgd,konevcny2018randomized,tang2019doublesqueeze,mishchenko2021intsgd,richtarik2021ef21,huang2022lower} have considered transmitting quantized or sparse information to the master node to save the communication costs. 
In the decentralized setting, existing compressed gradient methods can be classified into three main categories: (1) D(S)GD variants \cite{chocosgd,tang2018communication,singh2021squarm,koloskova2019decentralized,singh2022sparq,tang2019texttt}; (2) primal-dual like methods \cite{liu2021linear,li2021decentralized,lan2020communication,kovalev2021linearly,zhang2021innovation,yi2022communication}; (3) gradient tracking based algorithms \cite{liao2021compressed,zhao2022beer,song2021compressed,xu2022quantized,xiong2021quantized}. Compared to D(S)GD variants, the latter two types of methods can relieve from data heterogeneity. Specifically, D(S)GD variants cannot achieve exact convergence to the optimal solution under a constant stepsize even when the gradient variance goes to zero. By contrast, the methods considered in \cite{xiong2021quantized,michelusi2022finite,liu2021linear,li2021decentralized,liao2021compressed,song2021compressed,kovalev2021linearly,zhang2021innovation,xu2022quantized} enjoy linear convergence rate under smooth strongly convex objective functions with deterministic gradients.



In previous works, several decentralized stochastic gradient methods \cite{lian2017can,pu2021sharp,huang2021improving,pmlr-v80-tang18a,spiridonoff2020robust,pu2021distributed,alghunaim2021unified,yuan2021removing,koloskova2021improved,xin2021improved}, including the compressed D(S)GD type methods \cite{chocosgd,tang2018communication,koloskova2019decentralized,singh2022sparq,singh2021squarm,vogels2020practical,tang2019texttt}, have been shown to achieve linear speedup and enjoy the so-called ``asymptotic network independent property'' \cite{pu2020asymptotic}. In other words, the error term related to the network is negotiable after a finite number of iterations (transient time). 
 Such a property is desirable since it guarantees the number of iterations to reach a high accuracy level will not grow as fast as the size of the network increases. 
Currently, the shortest transient times achieved by decentralization stochastic gradient methods without communication compression are $\order{n/(1-\lambda_2)}$  for smooth strongly convex objective functions \cite{huang2021improving,yuan2021removing} and $\order{n^3/(1-\lambda_2)^2}$  for smooth nonconvex objective functions \cite{alghunaim2021unified}, respectively, where $(1-\lambda_2)$ defines the spectral gap of the mixing matrix. 
{ Recently, there is a line of works that manage to achieve the optimal iteration complexity in distributed stochastic optimization by integrating stochastic gradient accumulation and accelerated communication techniques such as Chebyshev Acceleration (CA) \cite{yuan2022revisiting,lu2021optimal}. 
In particular, the works in \cite{yuan2022revisiting} and \cite{lu2021optimal} obtain the optimal iteration complexity under smooth objective functions with and without the PL condition, respectively. 
Nevertheless, CA that relies on inner loops of multiple communication steps may not be communication-efficient, and sampling large batches for computing the stochastic gradients is not always practical \cite{xiao2023one}.}

Regarding compressed decentralized algorithms based on stochastic gradients, we compare the performance of those that achieve ``asymptotic network independent property'' in Table \ref{tab:kt} assuming that unbiased compressors are utilized. Note that in the literature, there are two commonly considered conditions on the compression compressors: (i) the signal noise ratio (SNR) between the compressed value and the original value is bounded by some constant $C> 0$, and the compressor is unbiased (see Assumption \ref{ass:ub_op} for details). We refer to such compressors as unbiased compressors; (ii) the corresponding SNR is bounded by $(1-\delta), \delta \in (0,1)$ (see Assumption \ref{ass:b_op}). We refer to such compressors as biased ones. Examples of the compressors that satisfy the above two assumptions can be found in \cite{beznosikov2020biased,safaryan2022uncertainty,xu2020compressed} and the references therein. Unbiased and biased compressors can be transformed into the other type under certain mechanisms \cite{horvath2021a}. 

Assuming biased compressors, Choco-SGD achieves a transient time that behaves as $\order{n/[(1-\lambda_2)^4\delta^2]}$ when the objective function $f$ is smooth strongly convex, and its transient time for smooth nonconvex objective function is $\order{n^3/[(1-\lambda_2)^8\delta^4]}$. SPARQ-SGD \cite{singh2022sparq} shares the same transient time with Choco-SGD as the method reduces to Choco-SGD without event-triggering and communication skip. The results listed in Table \ref{tab:kt} are based on an unbiased compressor satisfying Assumption \ref{ass:b_op} that has been transformed into a biased one following the mechanism in \cite{horvath2021a}.
However, among the communication compressed algorithms, only variants of D(S)GD have been shown to achieve the asymptotic network independent property. For example, LEAD \cite{liu2021linear} converges in the order of $\order{1/[(1-\lambda_2)^2 k]}$ (where $k$ counts the iteration number) when decreasing stepsizes are employed. Therefore, when $(1-\lambda_2)$ is small, more communication rounds are required for LEAD to achieve the same accuracy level compared to the centralized SGD method.

\subsection{Main Contribution}

The main contribution of this work is four-fold.

Firstly, we develop a new compressed decentralized stochastic gradient method, termed ``compressed exact diffusion with adaptive stepsizes (CEDAS)", and show the method asymptotically achieves comparable convergence rate as centralized SGD for { smooth objective functions under unbiased compression operators}.
{ In particular}, we characterize the transient time of CEDAS to reach the convergence rate of centralized SGD under unbiased compressors, which behaves as $\cO(n^3{C^6}/(1-\lambda_2)^4)$ for smooth nonconvex objective functions and $\cO(n{ C^3}/(1-\lambda_2)^{2})$ for smooth strongly convex objective functions (see Theorem \ref{thm:ncvx_KT} and Theorem \ref{thm:KT}).
Notably, the derived transient times for CEDAS are the shortest (with respect to the graph specifics) compared to the state-of-the-art methods
to the best of our knowledge (see Table \ref{tab:kt}). 
{In particular, CEDAS exhibits the shortest transient times  when
$C < \cO(1/(1 - \lambda_2)^2)$, which is common in practice. It is worth pointing out that the hidden constants such as $A^2$ and $\zeta^2$ illustrated in Table \ref{tab:kt} are usually unknown and can be large in some applications. 
}

{ Secondly, compared to the prior works, the obtained results for CEDAS do not require additional assumptions such as bounded second moments of the stochastic gradients or the bounded gradient dissimilarity condition (see Table \ref{tab:kt}).}

Thirdly, when no compression is performed, the derived transient times for CEDAS are consistent with those given in \cite{huang2021improving,alghunaim2021unified,yuan2021removing}, which are the shortest transient times so far among the decentralized stochastic gradient methods without communication compression. 

Finally, compared to the closely related algorithm LEAD \cite{liu2021linear} which has been shown to be successful for minimizing smooth strongly convex objectives, 
CEDAS is shown to work with both smooth nonconvex objective functions and smooth strongly convex objective functions.
Moreover, the convergence results for CEDAS are superior to LEAD under stochastic gradients. It is worth noting that obtaining the improved results is nontrivial because: 1) even without compression, analyzing CEDAS can be much more involved compared to studying DSGD variants; 2) the compression-related terms are nonlinear and require careful treatment; 3) dealing with nonconvexity is challenging. We have utilized various techniques involving constructing proper Lyapunov functions to demonstrate the results.

\begin{table}
	\centering
	\begin{tabular}{@{}cccc@{}}
		\toprule
		Algorithm                                   & $f$        & {\makecell[c]{Additional\\ Assumption} }           & Transient Time                                       \\ \midrule
		Choco-SGD \cite{koloskova2019decentralized} & NCVX       & {$A^2$-BG}             & $\frac{n^3C^4}{(1-\lambda_2)^8}$                     \\
		Choco-SGD \cite{chocosgd}                   & SCVX       & {$A^2$-BG}             & $\frac{nC^2}{(1-\lambda_2)^4}$                       \\
		DeepSqueeze \cite{tang2019texttt}           & NCVX       & {$\zeta^2$-BGD} & $\frac{n^3C^6}{(1-\lambda_2)^{12}}$                  \\
		SQuARM-SGD \cite{singh2021squarm}           & NCVX       & {$A^2$-BG}             & $\frac{n^3C^4}{(1-\lambda_2)^8}$                     \\
		SPARQ-SGD \cite{singh2022sparq}             & $f_i$ SCVX & {$A^2$-BG}             & $\frac{nC^2}{(1-\lambda_2)^4}$                       \\
		SPARQ-SGD \cite{singh2022sparq}             & NCVX       & {$A^2$-BG}           & $\frac{n^3C^4}{(1-\lambda_2)^8}$                     \\
		This work                                   & SCVX       & { /}                 & $\boldsymbol{\frac{n{C^3}}{(1-\lambda_2)^2}}$      \\
		This work                                   & NCVX       & {\makecell[c]{/}}                 & $\boldsymbol{\frac{n^3 { C^6}}{(1-\lambda_2)^4}}$ \\ \bottomrule
		\end{tabular}
	\caption{Comparison among the existing compressed decentralized stochastic gradient methods with the asymptotic network independent property. The transient times are presented by hiding some constants independent of $(1-\lambda_2)$, n, and $C$, where $C$ is the compressor-related parameter (see Subsection \ref{subsec:ass}) and $(1-\lambda_2)$ defines the spectral gap of the mixing matrix. Each $f_i$ is assumed to be a smooth function. NCVX stands for nonconvex functions and SCVX represents strongly convex functions. 
 	For the works that assume biased compression operators, we transform the unbiased compressors into biased ones and compute the compressor-related constant correspondingly according to \cite{horvath2021a}. { 
	{ The notation $A^2$-BG stands for requiring bounded second moments of the stochastic gradients}, i.e., $\E_{\xi_i}\brki{\normi{\nabla f_i(x;\xi_i)}^2}\leq A^2$. Such a condition is more restrictive than { the $\zeta^2$-BGD condition, which stands for the  bounded gradient
dissimilarity assumption that $\frac{1}{n}\sumn{\normi{\nabla f(x) - \nabla f_i(x)}^2}\leq \zeta^2$}.} }
	\label{tab:kt}
\end{table}

\subsection{Notation}
\label{sec:notations}

In this paper, we use column vectors by default. Let $x_{i,k}\in\R^p$ denote the local copy of agent $i$ at the $k-$th iteration. For the ease of presentation, we use bold lowercase letters and capital letters to denote stacked variables. For example, 
\begin{align*}
	\x_k&:= \prt{x_{1,k}, x_{2,k},\cdots, x_{n,k}}^{\T} \in \R^{n\times p}\\
	\nabla F(\x_k) &:= \prt{\nabla f_1(x_{1,k}), \cdots, \nabla f_n(x_{n,k})}^{\T}\in\R^{n\times p}.
\end{align*} 

We use $\bar{x}\in\R^p$ to denote the averaged (among agents) variables, e.g., $\bar{x}_k = \sumn x_{i,k}/n$ is defined as the average of all the $n$ agents' solutions at the $k-$iteration. Notation $\inpro{a,b}$ represents the inner product for two vectors $a, b\in\R^p$, while the inner product $\inpro{A,B}$ for two matrices $A,B\in\R^{n\times p}$ is defined as $\inpro{A,B}:= \sum_{i=1}^n \inpro{A_i,B_i}$, where $A_i$ stands for the $i-$row of $A$.

\subsection{Organization}

The remaining parts of this paper are organized as follows. In Section \ref{sec:setup}, we introduce the standing assumptions and the CEDAS algorithm. We then conduct the convergence analysis under smooth strongly convex objective functions in Section \ref{sec:scvx}. In Section \ref{sec:ncvx}, we perform the analysis for smooth nonconvex objective functions. Section \ref{sec:exp} provides numerical examples that corroborate the theoretical results, { and Section \ref{sec:conc} concludes this paper}.

\section{Setup}
\label{sec:setup}

In this section, we first introduce the standing assumptions in Subsection \ref{subsec:ass} with some necessary discussions. 
Then we present the new algorithm, termed ``compressed exact diffusion with adaptive stepsizes (CEDAS)", in Subsection \ref{subsec:lead_ds} along with some preliminary analysis { that hold regardless of the convexity condition in Subsection \ref{subsec:pre}. In particular, Fig. \ref{fig:cedas} in Section \ref{subsec:pre} presents the roadmap of the analysis}. 

\subsection{Assumptions}
\label{subsec:ass}
In this part, we introduce the standing assumptions for this work. Assumptions \ref{ass:W}, \ref{ass:sgrad}, \ref{ass:ub_op}, and \ref{ass:b_op} are related to the network structure, stochastic gradients and compression operators, respectively. Regarding the objective functions, we { first consider general smooth objective functions in Section \ref{sec:ncvx} with Assumption \ref{ass:smooth} only and then study smooth strongly convex functions in Section \ref{sec:scvx} under both Assumptions \ref{ass:fi} and \ref{ass:smooth}.} 

We start with stating the assumption regarding the multi-agent network structure. Suppose the agents are connected over a network $\mathcal{W} = (\mathcal{V}, \mathcal{E})$, where $\mathcal{V} = [n]$ represents the sets of nodes (agents), and $\mathcal{E}\subset \mathcal{V}\times \mathcal{V}$ denotes the set of edges linking different nodes in the network. We also denote $\cN_i=\crk{j| (i,j)\in \mathcal{E}}$ the set of neighbors of agent $i$. The matrix $W = (w_{ij})\in\R^{n\times n}$ is the mixing matrix compliant to the network $\mathcal{W}$. In particular, we assume the following condition on the network $\mathcal{W}$ and matrix $W$.
\begin{assumption}\label{ass:W}
	The graph $\mathcal{W}$ is undirected and strongly connected. There exists a link from $i$ and $j$ ($i\neq j$) in $\mathcal{W}$ if and only if $w_{ij},w_{ji}>0$; otherwise, $w_{ij}=w_{ji}=0$. The mixing matrix $W$ is {positive semidefinite \footnote{The positive semidefiniteness can be easily satisfied since we can choose $(I+W)/2$ if $W$ is not positive semidefinite.}}, nonnegative, symmetric and stochastic ($W\1 =\1$). 
\end{assumption}

Assumption \ref{ass:W} is common in the decentralized optimization literature; see, e.g., \cite{lian2017can,pu2021sharp,yuan2018exact}. The conditions imply that the eigenvalues of $W$, denoted as $\lambda_1\geq \lambda_2\geq \cdots\geq\lambda_n$, lie in the range of ${ [0},1]$. The term $(1-\lambda_2)$ is called the spectral gap of the graph/mixing matrix, which generally gets closer to $0$ when the connectivity of $\mathcal{W}$ is worse \cite{nedic2018network}. 

Regarding the stochastic gradients, we consider the following standard assumption. 
\begin{assumption}
	\label{ass:sgrad}
	For all iteration $k\ge 0$,  each agent $i$ is able to obtain noisy gradient $\nabla f_i(x_{i,k};\xi_{i,k})$ given $x_{i,k}$, where
	each random vector $\xi_{i,k}\in\R^q$ is independent across $i\in\mathcal{N}$. In addition, for some $\sigma>0$,
	\begin{equation}
		\label{condition: gradient samples}
		\begin{split}
			& \mathbb{E}[\nabla f_i(x_{i,k};\xi_{i,k})\mid x_{i,k}] =  \nabla f_i(x_{i,k}),\\
			& \mathbb{E}[\|\nabla f_i(x_{i,k};\xi_{i,k})-\nabla f_i(x_{i,k})\|^2\mid x_{i,k}] \le  \sigma^2.
		\end{split}
	\end{equation}
\end{assumption}

Stochastic gradients are common in machine learning problems. For example, when each agent $i$ randomly samples a minibatch of data points from its local dataset with replacement at every iteration and evaluate the gradient on the minibatch, an unbiased noisy gradient that is independent across the agents can be obtained. 

We now introduce the conditions on the compression operators. If a compressor $\cC$ satisfies Assumption \ref{ass:ub_op}, we denote $\cC\in \bU(C)$ for simplicity. Similarly, if $\cC$ satisfies Assumption \ref{ass:b_op}, we write $\cC\in\bB(\delta)$.

\begin{assumption}
	\label{ass:ub_op}
    The compression operator $\compress{\cdot}:\R^p\rightarrow\R^p$ is unbiased, i.e., $\condE{\compress{x}}{x} = x$ and satisfies 
	\begin{equation}
		\label{eq:com_C}
		\condE{\norm{\compress{x} - x}^2}{x}\leq C\norm{x}^2,\quad C>0,\forall x\in\R^p.
	\end{equation}
\end{assumption}

\begin{assumption}
    \label{ass:b_op}
    The compression operator $\compress{\cdot}:\R^p\rightarrow\R^p$ satisfies 
    \begin{equation}
        \label{eq:com_delta}
        \condE{\norm{\compress{x} - x}^2}{x} \leq (1-\delta)\norm{x}^2, \quad \delta\in(0,1],\forall x\in\R^p.
    \end{equation}
\end{assumption}
The above two types of compressors can be transformed between each other in the following way.
On one hand, note that given $\cC\in\bU(C)$, we can construct $\cC'\in \bB(1/(C+1))$ with $\cC':= \cC/(C+1)$. Therefore, an unbiased compressor can be transformed into a biased one.
On the other hand, Lemma \ref{lem:b2ub} below introduces a mechanism for constructing an unbiased compressor from any biased compressor. 
Such an idea first appears in \cite{horvath2021a}.

\begin{lemma}
    \label{lem:b2ub}
    For any compressor $\mathcal{C}_1\in\bB(\delta_1)$, we can choose a compressor $\mathcal{C}_2\in\bU(C_2)$ so that an introduced compressor $\cC: \R^p \rightarrow \R^p$ defined by $\compress{x} := \cC_1(x) + \cC_2\prt{x - \cC_1(x)}$
    satisfies Assumption \ref{ass:ub_op} with $C = C_2(1-\delta_1)$.
\end{lemma}

\begin{remark}
	The mechanism in Lemma \ref{lem:b2ub} allows users to apply algorithms that were not compatible with Assumption \ref{ass:b_op} originally.  The price to pay is more computations related to compressing $(x - \cC_1(x))$ and sending more bits. However, the compression parameter is decreased from $C_2$ to $C_2(1-\delta_1)$ compared to using $\cC_2$ directly. For the choice of $\cC_2$, it is preferable to choose $\cC_2$ with similar compression complexity as $\cC_1$ according to \cite{horvath2021a}. This would at most double the bits to send per iteration. 
\end{remark}


Assumptions \ref{ass:fi} and \ref{ass:smooth} below formally define strongly convex objective functions and smooth objective functions, respectively.

\begin{assumption}\label{ass:fi}
	The average function $f= \frac{1}{n}\sum_{i=1}^n f_i:\R^p\rightarrow \R$ is $\mu$-strongly convex, i.e., $\forall x, x'\in\R^p$,
	\begin{align*}
		&\langle\nabla f(x)-\nabla f(x'), x-x'\rangle\geq \mu \Vert x-x'\Vert^2.
	\end{align*}
\end{assumption}

\begin{assumption}
	\label{ass:smooth}
	Each $f_i:\R^p\rightarrow \R$ has $L$-Lipschitz continuous gradients,
    i.e., $\forall x, x'\in\R^p$, $\Vert \nabla f_i(x)-\nabla f_i(x')\Vert\leq L\Vert x-x'\Vert$, $\forall i$.
\end{assumption}

Define $\finf:= { \inf_{x\in\R^p} f(x)}$, then $f(x)\geq \finf, \forall x\in\R^p$.



In summary, the above assumptions are common and standard. In particular, Lemma \ref{lem:b2ub} provides a mechanism that generalizes the applicability of the proposed CEDAS algorithm introduced in the next section.

\subsection{Algorithm}
\label{subsec:lead_ds}

In this part, we introduce the CEDAS algorithm and discuss the strategies for analyzing CEDAS in light of the previous work \cite{huang2021improving}. 
The procedures of CEDAS are stated in Algorithm \ref{alg:lead_ds}.
\begin{algorithm}
	\caption{Compressed exact diffusion with adaptive stepsizes (CEDAS)}
	\label{alg:lead_ds}
	\begin{algorithmic}[1]
		\Require Stepsizes ${\eta_k}{ = \theta/[\mu(k + m)]}$ { or $\eta_k = \eta$}, parameters $\gamma$ and $\alpha$, and initial values $x_{i,-1}, h_{i,0}, i\in[n]$.
		\For{Agent $i$ in parallel}
		\State $\ul{d}_{i,0} = 0$
		\State $(h_w)_{i, 0} = \sum_{j\in \cN_i\cup\crk{i}}w_{ij}h_{j,0}$
		\State Compute $\nabla f_i(x_{i,-1};\xi_{i, -1})$ and $x_{i,0} = x_{i,-1} - \eta_{-1} \nabla f_i(x_{i,-1};\xi_{i, -1})$
		\EndFor
		\For{$k = 0,1,\cdots, K-1$, agent $i$ in parallel}
		\State Compute $\nabla f_i(x_{i,k};\xi_{i,k})$
		\State $y_{i,k} = x_{i,k} - \eta_k\nabla f_i(x_{i,k};\xi_{i,k}) - \ul{d}_{i,k}$ \label{line:y}
		\State { Obtain} $(\hat{y}_{i,k}, (\hat{y}_w)_{i,k}, h_{i,k+1}, (h_w)_{i,k + 1})$ by querying $\text{COMM} (y_{i,k}, h_{i,k}, (h_w)_{i,k},\alpha)$ \label{line:comm}
		\State $\ul{d}_{i,k + 1} = \ul{d}_{i,k} + \frac{\gamma}{2}\prt{\hat{y}_{i,k} - (\hat{y}_w)_{i,k}}$\label{line:d}
		\State $x_{i,k + 1} = x_{i,k} - \eta_k\nabla f_i(x_{i,k};\xi_{i,k}) - \ul{d}_{i, k + 1}$\label{line:x}
		\EndFor
		\State Output $x_{i, K}$.
	\end{algorithmic}
\end{algorithm}

{
\floatname{algorithm}{Procedure}
\setcounter{algorithm}{0}
\begin{algorithm}
	\caption{COMM($y_i, h_i, (h_w)_i, \alpha$)}
	\label{alg:comm}
	\begin{algorithmic}[1]
		\State $q_i = \compress{y_i - h_i}$
		\State $\hat{y}_i^+ = h_i + q_i$
		\For{Neighbors $j\in \cN_i$}
		\State Send $q_i$ and receive $q_j$
		\State $(\hat{y}_w)_i^+ = (h_w)_i + \sum_{j\in \cN_i\cup\crk{i}}w_{ij} q_j$
		\State $h_i^+ = (1-\alpha) h_i + \alpha \hat{y}_i^+$
		\State $(h_w)_i^+ = (1-\alpha)(h_w)_i + \alpha (\hat{y}_w)_i^+$
		\EndFor
		\State Output $\hat{y}_i^+, (\hat{y}_w)_i^+, h_i^+, (h_w)_i^+$
	\end{algorithmic}
\end{algorithm}
}

We first discuss the intuitive idea for the CEDAS algorithm. 
With the notations in Subsection \ref{sec:notations}, Problem \eqref{eq:P} can be equivalently written as
\begin{equation}
	\label{eq:P_cons}
	\min_{\x\in\R^{n\times p}} \sum_{i=1}^nf_i(x_i), \ (I-W)\x = \0, \ \x = (x_1, x_2,\cdots, x_n)^{\T}.
\end{equation}
As mentioned, for example, in \cite{huang2021improving,yuan2018exact,xu2021distributed}, we can perform the following primal-dual-like update to solve \eqref{eq:P_cons}:
\begin{subequations}
	\label{eq:pd}
	\begin{align}
		x_{i, k + 1} &= \sum_{j=1}^nw_{ij} \prt{x_{j,k} - \eta_k \nabla f_i(x_{i,k};\xi_{i,k})} - \tilde{d}_{i,k}\label{eq:pd1}\\
		\tilde{d}_{i,k+1} &= \tilde{d}_{i,k} + x_{i,k+1} - \sum_{j=1}^n w_{ij}x_{i,k +1}.\label{eq:pd2}
	\end{align}
\end{subequations}
The CEDAS algorithm can be viewed as equipping \eqref{eq:pd} with communication compression, where Procedure \ref{alg:comm} combines communication and compression for communication efficiency. More specifically, Line \ref{line:y} in Algorithm \ref{alg:lead_ds} is corresponding to the update \eqref{eq:pd1} without mixing among the agents. Line \ref{line:comm} produces the compressed version of $y_{i,k}$ as $\hat{y}_{i,k}$ and the mixed version of $\hat{y}_{i,k}$ as $(\hat{y}_w)_{i,k}$ in CEDAS. This procedure, which previously appears in \cite{liu2021linear}, indicates that $(\hat{y}_w)_{i,k} = \sum_{j=1}^n w_{ij}\hat{y}_{j,k}$ and $(h_w)_{i,k} = \sum_{j=1}^nw_{ij}h_{j,k}$. Then, Line \ref{line:d} performs the same step as \eqref{eq:pd2} based on the compressed information with an additional parameter $\gamma$. Such a parameter is introduced to control the so-called consensus error. In particular, Line \ref{line:d} is essentially performing $\ul{d}_{i, k + 1} = \ul{d}_{i,k} + \frac{\gamma}{2}\prti{\hat{y}_{i,k} - \sum_{j=1}^nw_{i,j}\hat{y}_{j,k}},$
which is similar to \eqref{eq:pd2}. Finally, Line \ref{line:x} (together with Line \ref{line:y}) performs an update similar to \eqref{eq:pd1} with mixed information. Note in procedure \ref{alg:comm}, $q_i = \compress{y_i - h_i}$ is the only variable to be transmitted by agent $i$.


Compared with the LEAD algorithm (Algorithm \ref{alg:lead} in Appendix \ref{app:lead}), 
CEDAS employs diminishing stepsizes so that the expected error $1/n\sum_{i=1}^n\E[\normi{x_{i,k} - x^*}^2]$ decreases to $0$ at an order optimal $\mathcal{O}_k\prti{{1}/{({ n}k)}}$ rate with stochastic gradients under strongly convex objective functions. Note that the update of the term $a_{i,k}$ in LEAD involves computing $1/\eta$ which is not compatible with diminishing stepsizes $\crki{\eta_k}$ as $\eta_k$ goes to zero. Therefore, we consider a different update for the corresponding term $d_{i,k}$ in Algorithm \ref{alg:lead_ds}. 



\subsection{Preliminary Analysis}
\label{subsec:pre}
The compact form of Algorithm \ref{alg:lead_ds} is given in \eqref{eq:lead_ds} below, based on which we perform some preliminary analysis on CEDAS. $G_k\in\R^{n\times p}$ denotes the compact form of stochastic gradient.
\begin{subequations}
    \label{eq:lead_ds}
    \begin{align}
        \y_k &= \x_k - \eta_k G_k - \ul{D}_k\\
        \ul{D}_{k + 1} 
		& = \ul{D}_k + \frac{\gamma}{2}(I-W)\prt{H_k + \compress{\y_k - H_k}}\label{eq:Dk0}\\
        \x_{k + 1} &= \x_k - \eta_k G_k - \ul{D}_{k + 1}\label{eq:xk0}.
    \end{align}
\end{subequations}



Inspired by \cite{huang2021improving,liu2021linear}, we introduce the compression error $E_k := \compress{\y_k - H_k} -(\y_k - H_k)$ and a new mixing matrix $\tW := I - \frac{\gamma}{2}(I-W)$. It follows that
\begin{align}
    \ul{D}_{k + 1} 
    &= \tW \ul{D}_k + (I-\tW) E_k + (I-\tW)(\x_k - \eta_k G_k)\label{eq:Dk1}.
\end{align}
From \eqref{eq:Dk0}, we have $\ul{D}_k\in \Span{I-W}$ for any $k\geq 0$ with initialization $\ul{D}_0 \in \Span{I-W}$. Noting the relation $I-\tW = \frac{\gamma}{2}(I-W)$, we obtain Lemma \ref{lem:tW} which directly results from the definition of $\tW$ for $\gamma\in(0,1)$. The proof is omitted.



\begin{lemma}
    \label{lem:tW}
    Let Assumption \ref{ass:W} hold and $\gamma \in (0, 1)$. We have 
    \begin{enumerate}[(a)]
        \item $\tW = I -\frac{\gamma}{2}(I-W)\in\R^{n\times n}$ is positive definite, symmetric, and stochastic;
        \item $\Span{I-\tW} = \Span{I-W}$;
        \item Let $\lambda_1\geq \lambda_2\geq \cdots\geq \lambda_n$ be the eigenvalues of $W$, then $\tilde{\lambda}_i = 1 - \frac{\gamma}{2}(1-\lambda_i), i=1,2,\cdots, n$ are the eigenvalues of $\tW$ and $1=\tlambda_1>\tlambda_2\geq \tlambda_3\geq \cdots\geq \tlambda_n\geq { 1 - \gamma / 2}$; 
        \item $\tW$ and $(I-\tW)$ commute.\label{ite:commute}
    \end{enumerate}
\end{lemma}

We now introduce new iterates $\crk{\s_k}$ to facilitate the analysis. Similar techniques can be found, e.g., in \cite{xu2021distributed,huang2021improving,yuan2018exact}. 
\begin{equation}
    \label{eq:Dksk}
    \ul{D}_k = (I-\tW)^{\frac{1}{2}} \s_k.
\end{equation}

With \eqref{eq:Dksk} and Lemma \ref{lem:tW} in hand, we discuss how to obtain the updates in terms of $\s_k$ and $\x_k$. Invoking \eqref{ite:commute} in Lemma \ref{lem:tW}, equation \eqref{eq:Dk1} becomes 
\begin{equation}
	\label{eq:sk}
    \begin{aligned}
		&(I-\tW)^{\frac{1}{2}} \s_{k + 1} = (I-\tW)^{\frac{1}{2}}\tW \s_k + (I-\tW) E_k\\
		&\quad + (I-\tW)(\x_k - \eta_k G_k).
	\end{aligned}
\end{equation}

Combining \eqref{eq:xk0} and \eqref{eq:Dk1} leads to 
\begin{align}
    \x_{k + 1}
    &= \tW (\x_k - \eta_k G_k)-\tW (I-\tW)^{\frac{1}{2}}\s_k - (I-\tW) E_k\label{eq:xk1}.
\end{align}

Let $\tV := (I-\tW)^{1/2}$. We are ready to derive the following recursions for the iterates $\crk{\x_k, \s_k}$ according to \eqref{eq:sk} and \eqref{eq:xk1}:
\begin{subequations}
    \label{eq:edas_c}
    \begin{align}
        \x_{k + 1} &= \tW (\x_k - \eta_k G_k)-\tW \tV\s_k - (I-\tW) E_k\label{eq:edas_cx}\\
        \s_{k + 1} &= \tW \s_k + \tV \prt{\x_k - \eta_k G_k} + \tV E_k\label{eq:edas_cs}.
    \end{align}
\end{subequations}
\begin{remark}
    Relation \eqref{eq:edas_c} resembles that of EDAS \cite{huang2021improving} when the compression error $E_k = 0$ with the new mixing matrix $\tW$. However, analyzing the compression error term $E_k$ is not trivial as the compression operator is nonlinear. In particular, it prevents us from applying the results in \cite{huang2021improving} and calls for additional procedures; see Lemma \ref{lem:ncvx_yh}. 
\end{remark}




{ Fig. \ref{fig:cedas} illustrates the roadmap of the follow-up analysis. Lemmas \ref{lem:descent0}-\ref{lem:ncvx_yh} serve as the cornerstones for constructing the Lyapunov functions $\cL_k^v$ and $\cL_k^b$, which play the key role for obtaining the convergence results of Algorithm \ref{alg:lead_ds}. Lemma \ref{lem:lya_ncvx} establishes the approximate ``descent'' property of $\E\cL_k^v$ which yields the convergence of CEDAS in Theorem \ref{thm:ncvx_ng} for minimizing smooth nonconvex objective functions. The recursion of $\E\cL_k^v$ is further enhanced to be ``contractive'' with additional errors due to strong convexity, as detailed in Lemma \ref{lem:cLv_scvx}. In addition, Lemma \ref{lem:metric} demonstrates that the metric for evaluating the convergence for smooth strongly convex objective functions, the total expected error $\sumn\E\brki{\normi{x_{i,k} - x^*}^2}/n$, can be upper bounded by $2\E\cL_k^v/\mu + 2\E\brki{\normi{\ch_k}^2}/n$. Such a result guides us to further derive the upper bounds for $\E\cL_k^v$ and $\E\brki{\normi{\ch_k}^2}$ in Lemma \ref{lem:lya}. The convergence result of CEDAS for minimizing smooth strongly convex objective functions is then stated in Theorem \ref{thm:total1}.
\begin{figure}
	\centering
	\includegraphics[width=0.49\textwidth]{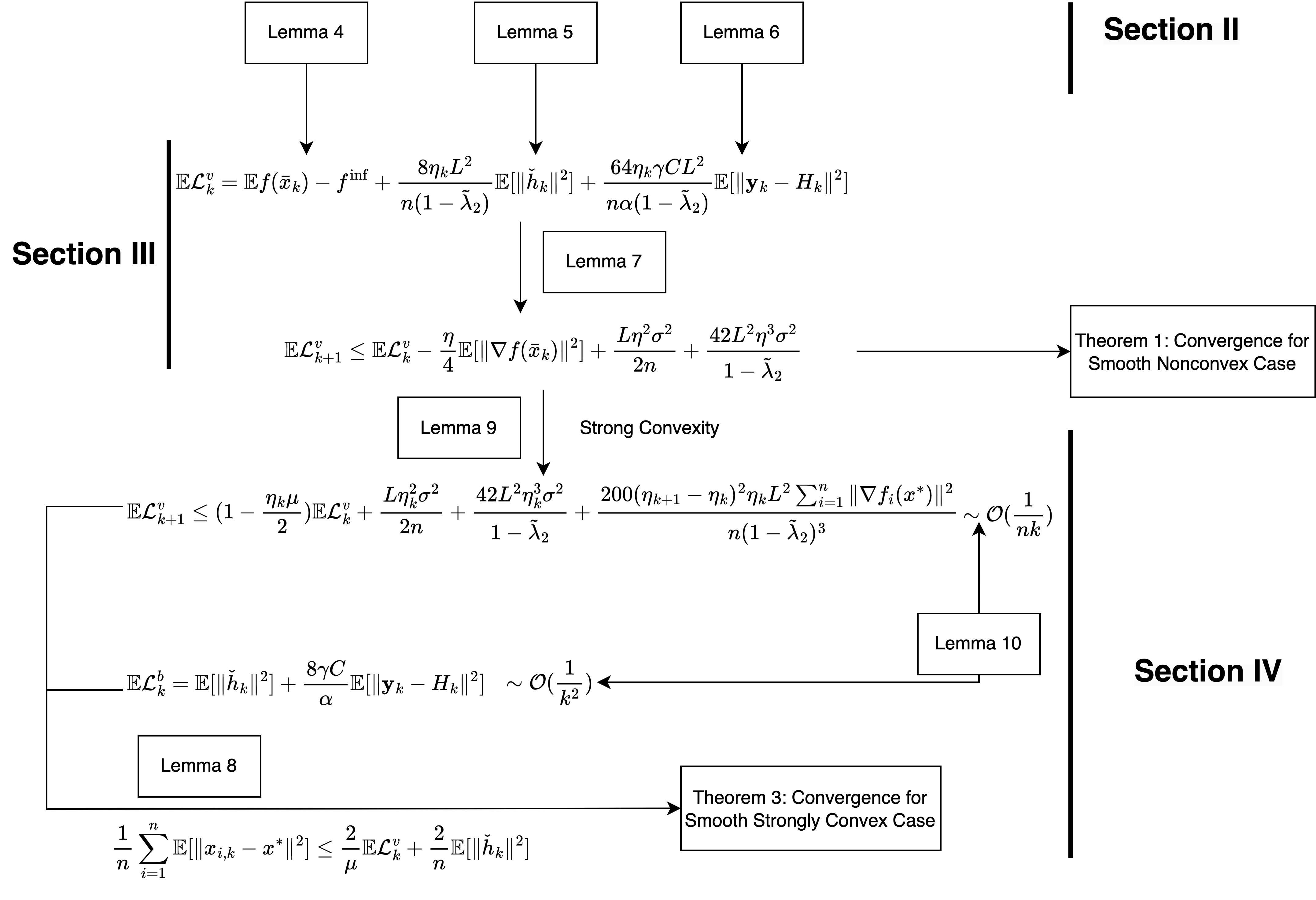}
	\caption{{ Roadmap of the analysis.}}
	\label{fig:cedas}
\end{figure}

}

{ Lemma \ref{lem:hk} utilizes the eigendecomposition $W = \1\1^{\T}/n + Q_1\Lambda_1Q_1^{\T}$ with $\Lambda_1:=\diag\prti{\lambda_2,\lambda_2,\ldots,\lambda_n}$ and introduces the transformed recursion corresponding to the consensus error. Similar ideas also appear in \cite{alghunaim2021unified}.}
\begin{lemma}
    \label{lem:hk}
    Let Assumption \ref{ass:W} hold. Denote {$\hat{\r}_k := \tV\s_k + \eta_k\nabla F(\1\bar{x}_k^{\T})$, $\hat{\g}_k:= \nabla F(\x_k) - G_k$, and} $\ch_k:= \prti{J_{L,l}Q_1^{\T}\x_k + J_{L,r}(I-\tLambda_1)^{-1/2}Q_1^{\T}\hat{\r}_k}$. Then we have 
	\begin{align*}
		&\ch_{k + 1} = P_1\ch_k + { P_1J_{L,l}}Q_1^{\T}\eta_k\hat{\g}_k + { (P_1 - I)J_{L,l}}Q_1^{\T}E_k\\
		&\quad + J_{L,r}(I-\tLambda_1)^{-\frac{1}{2}} Q_1^{\T}\eta_{{ k + 1}} \brk{\nabla F(\1\bar{x}_{k + 1}^{\T})-\nabla F(\1\bar{x}_k^{\T})}\\
		&\quad + P_1J_{L,l}Q_1^{\T}\eta_{ k}\brk{\nabla F(\1\bar{x}_k^{\T}) - \nabla F(\x_k)} \\
		&\quad + { J_{L,r} (I-\tLambda_1)^{-\frac{1}{2}}Q_1^{\T}(\eta_{k + 1} - \eta_k)\nabla F(\1\bar{x}_k^{\T})},
	\end{align*}
	where $\tLambda_1 := \diag\prti{\tlambda_2, \tlambda_3, \cdots, \tlambda_n}\in\R^{(n-1)\times (n-1)}.$
	The matrices $J_{L,l}, J_{L,r}\in\R^{2(n-1)\times (n-1)}$ are defined in \eqref{eq:js}, and $\normi{\x_k - \1\bar{x}_k}^2\leq 2\normi{\ch_k}^2$.

\end{lemma}

\begin{proof}
    See Appendix \ref{app:lem_hk}.
\end{proof}

\begin{remark}
	It is critical to introduce $\hat{\r}_k$. Suppose instead we consider $\s_k$, ${ \eta_k = \eta}$, and define $\ch_k':= (U_{L,l}\x_k + U_{L,r}\s_k)$, then similar derivation to Lemma{\ref{lem:hk}} would lead to $\ch_{k + 1}' = P_1\ch_k' + \eta P_1U_{L,l}\hat{\g}_k + \eta P_1 U_{L,l}\nabla F(\x_k) + (P_1 - I)U_{L,l}E_k,$
	which requires bounding $\eta^2\normi{\nabla F(\x_k)}^2$. This would make the convergence result suffer from data heterogeneity. By contrast, Lemma \ref{lem:hk} only requires considering $\eta^4\normi{\nabla F(\x_k)}^2$ when dealing with $(\nabla F(\1\bar{x}_{k + 1}^{\T}) - \nabla F(\1\bar{x}_k^{\T}))$. 
\end{remark}

\begin{lemma}
    \label{lem:descent0}
    Let Assumptions \ref{ass:W}, \ref{ass:sgrad}, and \ref{ass:smooth} hold and suppose $\eta_{ k}\leq 1/({ 2 }L)$, we have
    \begin{equation}
		\label{eq:descent0}
		\begin{aligned}
			\E f(\bar{x}_{k + 1}) - \finf &\leq \E f(\bar{x}_k) - \finf - \frac{\eta_{ k}}{2}\E\brk{\norm{\nabla f(\bar{x}_k)}^2}\\
			&\quad + \frac{\eta_{ k} L^2}{n}\E\brk{\norm{\ch_k}^2} + \frac{L\sigma^2\eta_{ k}^2}{2n}.
		\end{aligned}
	\end{equation}
\end{lemma}

\begin{proof}
    { See Appendix \ref{app:ncvx_lem_descent0}.}
\end{proof}

In light of Lemma \ref{lem:hk}, we study the recursion of $\E[\normi{\ch_k}^2]$, which gives rise to Lemma \ref{lem:ncvx_cons0}. 

\begin{lemma}
    \label{lem:ncvx_cons0}
    Suppose Assumptions \ref{ass:W}, \ref{ass:sgrad}, \ref{ass:ub_op}, and \ref{ass:smooth} hold, and let $\gamma\leq 1/2$, $\eta_k \leq \gamma (1-\lambda_2)/(32L)$.
    Then,
    \begin{align*}
        &\E\brk{\norm{\ch_{k + 1}}^2} \leq \frac{3 + \tlambda_2}{4}\E\brk{\norm{\ch_k}^2} + 4n\eta_k^2\sigma^2\\
		&+ 2\gamma C\E\brk{\norm{\y_k - H_k}^2} + \frac{36n\eta_k^4L^2}{(1-\tlambda_2)^2}\E\brk{\norm{\nabla f(\bar{x}_k)}^2}\\
		&{ + \frac{12(\eta_{k + 1} - \eta_k)^2}{(1-\tlambda_2)^2} \E\brk{\norm{\nabla F(\1\bar{x}_k^{\T})}^2}}.
    \end{align*}
\end{lemma}

\begin{proof}
    See Appendix \ref{app:ncvx_lem_cons0}.
\end{proof}

We next derive the recursion for $\E[\normi{\y_k - H_k}^2]$ in Lemma \ref{lem:ncvx_yh} below.
\begin{lemma}
    \label{lem:ncvx_yh}
    Suppose Assumptions \ref{ass:W}, \ref{ass:sgrad}, \ref{ass:ub_op}, and \ref{ass:smooth} hold. Let ${ \alpha < 1/[4(C+1)]}$,${\eta_k\leq \gamma/(3L)}$, and ${\gamma \leq \min\crki{1/2, \alpha  / 2}}$.
    We have 
	{
	\begin{align*}
		&\E\brk{\norm{\y_{k + 1} - H_{k + 1}}^2}\leq \prt{1 - \frac{\alpha }{3}}\E\brk{\norm{\y_k - H_k}^2}\\
		&\quad + \frac{124\gamma}{\alpha} \E\brk{\norm{\ch_k}^2}+ \frac{54\eta_k^2n }{\alpha}\E\brk{\norm{\nabla f(\bar{x}_k)}^2} \\
		&\quad + \frac{18(\eta_k - \eta_{k + 1})^2}{\alpha }\E\brk{\norm{\nabla F(\x_k)}^2}+ 6\eta_k^2n\sigma^2.
	\end{align*}
	}
\end{lemma}

\begin{proof}
    See Appendix \ref{app:ncvx_lem_yh0}.
\end{proof}

\section{Convergence Analysis: Nonconvex Case}
\label{sec:ncvx}

{ In this section, we consider the case when the objective function $f$ is smooth nonconvex. The primary step of analysis is to locate a Lyapunov function $\cL_k^v$, which maintains the approximate ``descent'' property stated in Lemma \ref{lem:lya_ncvx}.}
\begin{equation}
	\label{eq:lya_ncvx}
	\begin{aligned}
		\cL_k^v&:= f(\bar{x}_k) - \finf + v_1\eta_k \norm{\ch_k}^2 + v_2\eta_k \norm{\y_k - H_k}^2,\\
		v_1&:= \frac{8L^2}{n(1-\tlambda_2)}, \ v_2 := \frac{{ 64}\gamma C L^2}{n{\alpha} (1-\tlambda_2)}.
	\end{aligned}
\end{equation}

\begin{lemma}
    \label{lem:lya_ncvx}
    Suppose Assumptions \ref{ass:W}, \ref{ass:sgrad}, \ref{ass:ub_op}, and \ref{ass:smooth} hold. 
    { Let the parameters $\eta_k$, $\alpha$, and $\gamma$ satisfy
	\begin{align}
			&\eta_k = \eta \leq \frac{\gamma (1-\lambda_2)}{32 L},\label{eq:params_eta}\\
			&\gamma \leq \min\crk{\frac{1}{2}, \frac{\alpha^2 (1-\lambda_2)}{15872 C}, \frac{\alpha}{2}},\; \alpha \leq \frac{1}{4(C+1)}.\label{eq:params_ga}
		\end{align}
	}
    We have 
    \begin{align*}
		\E\cL_{k + 1}^v &\leq \E\cL_k^v - \frac{\eta}{4}\E\brk{\norm{\nabla f(\bar{x}_k)}^2} + \frac{L\eta^2\sigma^2}{2n} + \frac{42L^2\eta^3\sigma^2}{1-\tlambda_2}.
	\end{align*}

\end{lemma}

\begin{proof}
    See Appendix \ref{app:lya_ncvx}.
\end{proof}

\subsection{Convergence}

With the help of Lemma \ref{lem:lya_ncvx}, we can characterize the convergence of CEDAS for minimizing smooth nonconvex objective functions in Theorem \ref{thm:ncvx_ng} { and its asymptotic network independent behavior in Corollary \ref{cor:ncvx_rate}}. 

\begin{theorem}
    \label{thm:ncvx_ng}
    Suppose Assumptions \ref{ass:W}, \ref{ass:sgrad}, \ref{ass:ub_op}, and \ref{ass:smooth} hold. 
    { Let $\eta_k$ satisfy \eqref{eq:params_eta}. Set $\alpha$ and $\gamma$ to satisfy \eqref{eq:params_ga}.
	}
    We have 
	\begin{equation}
		\label{eq:ncvx_compl}
		\begin{aligned}
			&\frac{1}{K }\sum_{k=-1}^{K-{ 2}}\E\brk{\norm{\nabla f(\bar{x}_k)}^2} \leq \frac{4\prt{f(\bar{x}_{-1}) - \finf}}{\eta K} \\
			&\quad + \frac{{ 256}\gamma C L^2{\E\brk{\norm{\y_0 - H_0}^2}}}{n \alpha (1-\tlambda_2)K} 
			+ \frac{{ 6}L\eta\sigma^2}{n} + \frac{168\eta^2L^2\sigma^2}{1-\tlambda_2}\\
			&\quad + \frac{4L^2\norm{\x_{-1} - \1\bar{x}_{-1}^{\T}}^2}{nK} + \frac{32L^2\E\brk{\norm{\ch_0}^2}}{n(1-\tlambda_2)K}.
		\end{aligned}
	\end{equation}

\end{theorem}

\begin{proof}
	We first consider the term $\E f(\bar{x}_0) - \finf$. Similar to { the derivations in} Lemma \ref{lem:descent0}, letting $\eta \leq 1/L$ yields 
	\begin{equation}
		\label{eq:f0}
		\begin{aligned}
			&\E f(\bar{x}_0) - \finf \leq f(\bar{x}_{-1}) - \finf - \frac{\eta}{2}\norm{\nabla f(\bar{x}_{-1})}\\
			&\quad + \frac{\eta L^2 \norm{\x_{-1} - \1\bar{x}_{-1}^{\T}}^2}{n} + \frac{L\eta^2 \sigma^2}{2n}.
		\end{aligned}
	\end{equation}
	Taking the { sum} among $k=0,1,\cdots, K { -2}$ on both sides of the inequality in Lemma \ref{lem:lya_ncvx} yields 
	{
	\begin{equation}
		\label{eq:cL_sum}
		\begin{aligned}
			\sum_{k=0}^{K-2}\E\brk{\norm{\nabla f(\bar{x}_k)}^2} 
			&\leq \frac{4\prt{\E\cL_0^v - \E\cL_{K-1}^v}}{\eta}+ \frac{2L\eta (K-1)\sigma^2}{n}\\
			&\quad + \frac{168\eta^2L^2\sigma^2 (K-1)}{1-\tlambda_2}.
		\end{aligned}
	\end{equation}
	
	Combining \eqref{eq:f0} and \eqref{eq:cL_sum} and taking the average lead to the desired result.}

\end{proof}



By choosing a specific stepsize $\eta$ and { hiding the constants that are independent of $n$ and $(1-\lambda_2)$}, we obtain the convergence rate of CEDAS stated in Corollary \ref{cor:ncvx_rate} which behaves as $\orderi{1/\sqrt{nK}}$. Such a result is comparable to the centralized SGD method. 

\begin{corollary}
	\label{cor:ncvx_rate}
	Suppose Assumptions \ref{ass:W}, \ref{ass:sgrad}, \ref{ass:ub_op}, and \ref{ass:smooth} hold. Let{$\alpha$ and $\gamma$ satisfy \eqref{eq:params_ga}. Denote $\Delta_0:= f(\bar{x}_{-1}) - \finf$. Set the stepsize $\eta_k = \eta$ to be 
	\begin{align}
		\label{eq:eta_ncvx}
		\eta = \frac{1}{\sqrt{\frac{ L\sigma^2K}{n\Delta_0}} + \frac{32 L}{\gamma (1-\lambda_2)}}.
	\end{align}
	}
	
    
    We have 
	\begin{align*}
		&\frac{1}{K}\sum_{k=-1}^{K-{ 2}}\E\brk{\norm{\nabla f(\bar{x}_k)}^2} = \order{\frac{1}{\sqrt{nK}} + \frac{n}{(1-\tlambda_2)K}\right.\\
		&\quad\left. + \frac{\norm{\x_{-1}}^2+ \norm{H_0}^2}{n(1-\tlambda_2)K} + \frac{\norm{\nabla F(\x_{-1})}^2}{(1-\tlambda_2)^2 K^2}}.
	\end{align*}
\end{corollary}
\begin{proof}
	{ For $\eta$ satisfying \eqref{eq:eta_ncvx}, we have $\eta\leq \sqrt{n\Delta_0/(L\sigma^2 K)}$ and 
	\begin{equation}
		\label{eq:eta_inv}
		\begin{aligned}
			&\frac{\Delta_0}{\eta K} = \sqrt{\frac{\Delta_0 L \sigma^2}{n K}} + \frac{32 L\Delta_0}{\gamma (1-\lambda_2) K}.
		\end{aligned}
	\end{equation}
	}
	
	For the terms $\E\brki{\normi{\ch_0}^2}$ and $\E\brki{\normi{\y_0 - H_0}^2}$, we have 
	\begin{equation}
		\label{eq:initial}
		\begin{aligned}
			&\E\brk{\norm{\ch_0}^2} = \order{\normi{\x_{-1}}^2 + \frac{\eta^2\normi{\nabla F(\x_{-1})}^2}{1-\tlambda_2}},\\
			&\E\brk{\norm{\y_0 - H_0}^2} = \orderi{\norm{\x_{-1}}^2 + \eta^2\norm{\nabla F(\x_{-1})}^2 + \norm{H_0}^2}.
		\end{aligned}
	\end{equation}
	{ Substituting \eqref{eq:eta_inv} and \eqref{eq:initial} into \eqref{eq:ncvx_compl} leads to the desired result.}

\end{proof}




\subsection{Transient Time}

This part introduces the transient time of CEDAS for minimizing smooth objective functions, where the formal definition of transient time is as the following: 
\begin{equation}
	\label{def:transient_ncvx}
	\begin{aligned}
		K_T^{(ncvx)}
		&:= \inf_K\left\{\frac{1}{k}\sum_{t=0}^{k - 1}\E\brk{\norm{\nabla f(\bar{x}_t)}^2}\leq \order{\frac{1}{\sqrt{nk}}},\right.\\
		&\left.\quad \forall k\ge K. \right\}
	\end{aligned}
\end{equation}

In other words, the transient time denotes the least time for the convergence rate of Algorithm \ref{alg:lead_ds} to be dominated by the convergence rate of centralized SGD. 

The transient time of CEDAS for minimizing smooth nonconvex objective functions is stated in Theorem \ref{thm:ncvx_KT} below.
\begin{theorem}
    \label{thm:ncvx_KT}
    Suppose Assumptions \ref{ass:W}, \ref{ass:sgrad}, \ref{ass:ub_op}, and \ref{ass:smooth} hold.  Let{$\alpha$ and $\gamma$ to satisfy \eqref{eq:params_ga}. Set the stepsize $\eta_k = \eta$ to satisfy \eqref{eq:eta_ncvx}.
	}
    It takes 
    \begin{align*}
        &K_T^{(ncvx)} = \max\crk{\order{\frac{n^3}{\gamma^2 (1-\lambda_2)^2}},\right.\\
		&\quad\left. \order{\frac{\norm{\x_{-1}}^{4} + \norm{H_0}^4}{n\gamma^2 (1-\lambda_2)^2}}, \order{\prt{\frac{n \norm{\nabla F(\x_{-1})}^4}{\gamma^4(1-\lambda_2)^4}}^{\frac{1}{3}}} }
    \end{align*}
    for Algorithm \ref{alg:lead_ds} to achieve the asymptotic network independent convergence rate.
\end{theorem}

\begin{proof}
    Invoking the definition of $K_T^{(ncvx)}$ in \eqref{def:transient_ncvx} and Corollary \ref{cor:ncvx_rate}, we obtain the desired result.
\end{proof}

We next introduce a mild additional condition \eqref{eq:ncvx_initial} to simplify the transient time in Theorem \ref{thm:ncvx_KT}. { Denote $\beta_0:= 15872$.}

\begin{corollary}
    \label{cor:ncvx_ub}
    Let the conditions in Theorem \ref{thm:ncvx_KT} hold and further assume that for some $q>0$,
    \begin{equation}
		\label{eq:ncvx_initial}
		\begin{aligned}
			\norm{\x_{-1}}^2 &= \norm{H_0}^2 = \order{n^2},\ \norm{\nabla F(\x_{-1})}^2 = \order{n^2},\\
			1/(1-\lambda_2) &= \order{n^q}.
		    \end{aligned}
	\end{equation}
    Let $\alpha = { 1/[4(C+1)]}$ and $\gamma= { \alpha^2(1-\lambda_2)/(\beta_0 C)}$. { Set $\eta$ to satisfy \eqref{eq:eta_ncvx}.}
    Then the transient time of Algorithm \ref{alg:lead_ds} becomes
    \begin{align}
		\label{eq:kt_ncvx}
        K_T^{(ncvx)} = \order{\frac{n^3C^6}{(1-\lambda_2)^4}}.
    \end{align}
\end{corollary}

\begin{remark}
	Condition \eqref{eq:ncvx_initial} is mild. Indeed, we can initialize $x_{i,-1} = h_{i,0}=x_{-1}$ for any $i\in[n]$, then $\normi{\x_{-1}}^2 = \normi{H_0}^2 = n\normi{x_{-1}}^2$ and $\normi{\nabla F(\x_{-1})}^2 = \sum_{i=1}^n\normi{\nabla f_i(x_{-1})}^2$. Such an initialization satisfies condition \eqref{eq:ncvx_initial}.
	According to \cite{nedic2018network}, $1/(1-\lambda_2) = \orderi{n^2}$ is satisfied for any connected undirected graph if we use the Lazy Metropolis rule \eqref{eq:lazy} for constructing the mixing matrix $W$, i.e.,
	\begin{equation}
		\label{eq:lazy}
		w_{ij} = \begin{cases}
			\frac{1}{2\max\crk{\deg(i), \deg(j)}}, & i\in \cN_i\\
			1 - \sum_{j\in\cN_i} w_{ij}, & i = j\\
			0, & \text{ otherwise}
		\end{cases}.
	\end{equation}
\end{remark}


We next can derive the transient time for CEDAS when the compression operator is biased.

\begin{corollary}
	\label{cor:ncvx_b}
	Let the conditions in Theorem \ref{thm:ncvx_KT} hold. In particular, assume the compression operator satisfies Assumption \ref{ass:b_op}. If we further assume condition \eqref{eq:ncvx_initial} holds and apply the technique in Lemma \ref{lem:b2ub} with some unbiased compressor with parameter $C$, then we can choose {$\alpha = 1/\crki{4[C(1-\delta) + 1]}$}. Let {$\gamma = \alpha^2(1-\lambda_2)/[\beta_0 C(1-\delta) ]$}, and $\eta$ { satisfy \eqref{eq:eta_ncvx}.}
	Then the transient time of Algorithm \ref{alg:lead_ds} becomes
    \begin{align*}
        K_T^{(ncvx)} = \order{\frac{n^3C^6(1-\delta)^6}{(1-\lambda_2)^4}}.
    \end{align*}
\end{corollary}
Finally, we can obtain the transient time for CEDAS when there is no compression.
\begin{corollary}
	\label{cor:ncvx_KT_noC}
	Let the conditions in Theorem \ref{thm:ncvx_KT} hold and further assume condition \eqref{eq:ncvx_initial}. In addition, suppose there is no compression, i.e., $C = 0$. Choose $\alpha = 1/4$ and $\gamma= 1 / 8$. { Set $\eta$ to satisfy \eqref{eq:eta_ncvx}.}
	Then the transient time of Algorithm \ref{alg:lead_ds} becomes
	\begin{align*}
		K_T^{(ncvx)} = \order{\frac{n^3}{(1-\lambda_2)^2}}.
	\end{align*}
\end{corollary}
\begin{remark}
    Such a result is consistent to those given in \cite{alghunaim2021unified} when the objective functions are smooth and nonconvex.
\end{remark}

\section{Convergence Analysis: Strongly Convex Case}
\label{sec:scvx}
In this section, we present the convergence analysis for CEDAS when the objective function is strongly convex and smooth, in which case Problem \eqref{eq:P} has an optimal solution $x^*\in\R^p$. { To facilitate the reading, we first state the main results and then give the supporting lemmas in Subsection \ref{subsec:scvx_pre}. { Compared with the nonconvex case}, the convergence metric becomes the total expected error $\sumn\E\brki{\normi{x_{i,k} - x^*}^2}/n$, which can be upper bounded by $2\E\cL_k^v/\mu + 2\E\brki{\normi{\ch_k}^2}/n$ according to Lemma \ref{lem:metric}. Hence, the analysis focuses on deriving the upper bound for $\E\cL_k^v$ and $\E\brki{\normi{\ch_k}^2}$ as detailed in Lemma \ref{lem:lya}.
}

\subsection{Convergence}
\label{subsec:sublinear}


We state the convergence rate of CEDAS for the total expected error in Theorem \ref{thm:total1} below { when using decreasing stepsizes $\eta_k$ specified in \eqref{eq:etak}:
\begin{equation}
	\label{eq:etak}
	\eta_k = \frac{\theta}{\mu(k + m)}, \text{ for some }\theta > 0.
\end{equation}
}

\begin{theorem}
	\label{thm:total1}
	{Suppose Assumptions \ref{ass:W}, \ref{ass:sgrad}, \ref{ass:ub_op}, \ref{ass:fi}, and \ref{ass:smooth} hold. Let the stepsize policy be \eqref{eq:etak} { and set $\theta\geq 20$, $m\geq 48\theta\kappa/[\gamma(1-\lambda_2)]$, $\alpha\leq 1/[4(C+1)]$, and 
	\begin{equation}
		\begin{aligned}
			\gamma &\leq \min\crk{\frac{1}{2}, \frac{\alpha^2 (1-\lambda_2)}{15876 C}, \frac{\alpha }{2}}.
		\end{aligned}
	\end{equation}
	}} 
	
	We have for Algorithm \ref{alg:lead_ds} that
	{
	\begin{align*}
		&\frac{1}{n}\sumn \E\brk{\norm{x_{i,k} - x^*}^2} \leq \prt{1 - \frac{1-\tlambda_2}{6}}^k \frac{2\E\cL_0^b}{n}\\
		&\quad + \frac{8 L\theta\sigma^2}{n\mu^3(k + m)} + \frac{48\theta^2(14\kappa^2 + 3)\sigma^2}{\mu^2(1-\tlambda_2)(k + m)^2}\\
		&\quad + \frac{96\theta^2 (32 \kappa^2 + 9)\sumn\norm{\nabla f_i(x^*)}^2}{n\mu^2(1-\tlambda_2)^3 (k + m)^4} + \frac{70m^{\frac{\theta}{2}}\E\cL_0^v}{\mu(k + m)^{\frac{\theta}{2}}}\\
		&\quad + \frac{456\theta^3\kappa^2\sigma^2}{n\mu^2(1-\tlambda_2)(k + m)^3} + \frac{96768\theta^4 \kappa^3\sigma^2}{\mu^2(1-\tlambda_2)^2(k + m)^4}.
	\end{align*}
	}
\end{theorem}
\begin{proof}
	Substituting \eqref{eq:cLv_scvx1} and \eqref{eq:cons1} into \eqref{eq:metric} yields the desired result.
\end{proof}
\begin{remark}
	Theorem \ref{thm:total1} gives the convergence result of CEDAS under smooth and strongly convex objective functions assuming the compression operator is unbiased. Compared to the $\order{1/[(1-\lambda_2)^2 k]}$ convergence rate in \cite{liu2021linear}, we are able to show the linear speedup property, i.e., $\order{1/(nk)}$ asymptotic convergence rate of CEDAS. 
\end{remark}


\subsection{Transient Time}
\label{subsec:KT}
In this part, we estimate 
the transient time of CEDAS. We first give the formal definition of the transient time (for smooth strongly convex objective functions) in \eqref{def:transient}:

\begin{align}
	K_T^{(scvx)}&:=\inf_K\left\{\frac{1}{n}\sum_{i=1}^n\E\left[\Vert x_{i,k}-x^*\Vert^2\right]\le \order{\frac{1}{nk}},\right.\nonumber\\
	&\left.\quad \forall k\ge K. \right\}.\label{def:transient}
\end{align}


Theorem \ref{thm:KT} states the transient time of Algorithm \ref{alg:lead_ds}.
\begin{theorem}
	\label{thm:KT} 
	Let the conditions in { Theorem \ref{thm:total1}} hold. Then it takes 
	\small
	\begin{align*}
		&K_T^{(scvx)}= \max\crk{\order{\frac{n}{\gamma(1-\lambda_2)}}, \order{{\prt{\frac{\norm{\x_{-1}}^2}{[\gamma(1-\lambda_2)]^{\frac{\theta}{2}}}}^{\frac{2}{\theta - 2}} }},\right.\\
		&\left. \order{\frac{{\prti{\sumn\norm{\nabla f_i(x^*)}^2}^{\frac{1}{3}}}}{\gamma(1-\lambda_2)}}, \order{{\prt{\frac{\norm{H_{0}}^2}{[\gamma(1-\lambda_2)]^{\frac{\theta}{2}}}}^{\frac{2}{\theta - 2}}}} }
	\end{align*}\normalsize
	iterations for Algorithm \ref{alg:lead_ds} to reach the asymptotic network independent convergence rate, that is, when $k\geq K_T^{(scvx)}$, we have $\sumn \E\brki{\normi{x_{i,k} - x^*}^2}/n \leq \orderi{1/(nk)}$.
	
	
\end{theorem}



\begin{proof}
   See Appendix \ref{app:thm_KT}.

\end{proof}


Under some  mild additional  conditions, we can obtain a cleaner expression for the transient times of CEDAS with unbiased and biased compression operators respectively in the following corollaries.

\begin{corollary}
	\label{cor:KT_unbiased}
	Let the conditions in { Theorem \ref{thm:total1}} hold and initiate $h_{i,0} = x_{i,-1}, \forall i\in[n]$. If we further choose {$\alpha = 1/[4(C+1)]$ and $\gamma = \alpha^2(1-\lambda_2)/(\beta_0 C)$} and assume for some $q>0$,
    \begin{equation}
		\label{eq:graph_p}
		\begin{aligned}
			\sumn\norm{\nabla f_i(x^*)}^2 &= \order{n^3},\ {\norm{\x_{-1}}^2} = \order{n^2},\\
			 1/(1-\lambda_2) &= \order{n^q}.
		\end{aligned}
	\end{equation}

    Then 
	\begin{align}
        \label{eq:KT_ub_p}
		K_T^{(scvx)} = \order{\frac{nC^3}{(1-\lambda_2)^2}}.
	\end{align}

\end{corollary}


\begin{remark}
	The condition $\sumn\normi{\nabla f_i(x^*)}^2 = \order{n^3}$ holds for many problem settings, including linear regression and logistic regression. The restriction on the initial values $x_{i,-1}, i\in[n]$ is also mild. It aims to simplify the formula of the transient time. Indeed, we can initialize all the agents with the same solution to satisfy ${\normi{\x_{-1}}^2} = \orderi{n}$. 
\end{remark}

When the compression operator is biased (satisfying Assumption \ref{ass:b_op}), we can also derive the transient time for CEDAS using the technique in Lemma \ref{lem:b2ub}. 
\begin{corollary}
	\label{cor:KT_biased}
	Let the conditions in { Theorem \ref{thm:total1}} hold except Assumption \ref{ass:ub_op} and initiate $h_{i,0} = x_{i,-1},\forall i\in[n]$. In addition, assume the compression operator satisfies Assumption \ref{ass:b_op}, and we apply the technique in Lemma \ref{lem:b2ub} with some unbiased compressor with parameter $C$. Under condition \eqref{eq:graph_p}, letting {$\alpha = 1/\crki{4[C(1-\delta) + 1]}$ and $\gamma = \alpha^2 (1-\lambda_2)/[\beta_0C(1-\delta)]$} yields the transient time of Algorithm \ref{alg:lead_ds}:
	\begin{equation}
		\label{eq:KT_b}
		K_T^{(scvx)} = \order{\frac{n{C^3(1-\delta)^3}}{(1-\lambda_2)^2}}.
	\end{equation}

\end{corollary}

We can also calculate the transient time for CEDAS when there is no compression.
\begin{corollary}
	\label{cor:KT_noC}
	Suppose Assumptions \ref{ass:W}, \ref{ass:sgrad}, and \ref{ass:fi} hold. We further assume there is no compression, i.e., $C=0$. Let the stepsize policy be \eqref{eq:etak} and condition \eqref{eq:graph_p} hold. We set $\alpha= 1/4$, $\gamma = 1/8$, $\theta \geq { 20}$, and {$m\geq 48\theta\kappa/[\gamma(1-\lambda_2)]$}.
    Then, the transient time of CEDAS reduces to 
    \begin{align*}
		K_T^{(scvx)} = \order{\frac{n}{1-\lambda_2}}.
    \end{align*}
\end{corollary}

\begin{remark}
	When there is no compression error, the transient time of Algorithm \ref{alg:lead_ds} is similar to that of EDAS \cite{huang2021improving}. 
 If we further assume condition \eqref{eq:graph_p} holds, the transient time becomes $K_T = \orderi{n/(1-\lambda_2)}$. 
 This is consistent with our discussions in Section \ref{subsec:lead_ds}.
\end{remark}

\subsection{Supporting Lemmas}
\label{subsec:scvx_pre}


{ 
Denote $f^*:= f(x^*)$. Lemma \ref{lem:metric} below guides us to bound the corresponding terms when the objective function $f$ is smooth strongly convex. 

\begin{lemma}
	\label{lem:metric} 
	Let Assumptions \ref{ass:fi} and \ref{ass:smooth} hold. Set $\gamma \leq 1/2$. We have for any $k\geq 0$ that
	\begin{align}
		\frac{1}{n}\sumn\norm{x_{i,k} - x^*}^2
		&\leq \frac{2}{\mu}\cL_k^v + \frac{2}{n}\norm{\ch_k}^2.\label{eq:metric}
	\end{align}
\end{lemma}

\begin{proof}
	See Appendix \ref{app:lem_metric}.
\end{proof}

Lemma \ref{lem:cLv_scvx} improves upon Lemma \ref{lem:lya_ncvx} with the help of strong convexity. 

\begin{lemma}
	\label{lem:cLv_scvx} 
        Let the conditions in Theorem \ref{thm:total1} hold.
	We have 
	\begin{equation}
		\label{eq:cLv_scvx} 
		\begin{aligned}
			&\E\cL_{k + 1}^v \leq \prt{1 - \frac{\eta_k \mu}{2}}\E\cL_k^v + \frac{L\eta_k^2 \sigma^2}{2n} + \frac{42\eta_k^3 L^2\sigma^2}{1-\tlambda_2}\\
			&\quad + \frac{200(\eta_k - \eta_{k + 1})^2\eta_k L^2\sumn\norm{\nabla f_i(x^*)}^2}{n(1-\tlambda_2)^3}.
		\end{aligned}
	\end{equation}

\end{lemma}

\begin{proof}
	See Appendix \ref{app:lem_cLv_scvx}.
\end{proof}

}

We next construct { an additional} Lyapunov functions $\E\cL_k^b$ to { upper bound the error term $\E\brki{\normi{\ch_k}^2}$}. We will show in Lemma \ref{lem:lya} that $\E\cL_k^{ v}\sim\order{1/k}$ and $\E\cL_k^b\sim\order{1/k^2}$ when using decreasing stepsizes $\eta_k$ specified in \eqref{eq:etak}.
The Lyapunov function $\E\cL_k^b$ takes the form stated below which guides us to bound {$\E\brki{\normi{\ch_k}^2}$}:
\begin{align}
	\cL_k^b &:= \norm{{ \ch_k}}^2 + b_1\norm{\y_k - H_k}^2 \label{eq:lya2_cd},
\end{align}
where $b_1= { 8}\gamma C / \alpha$.

Lemma \ref{lem:lya} proves a $\orderi{1/{ n}k}$ bound for $\E\cL_k^{ v}$ and { a $\orderi{n/[(1-\tlambda_2)k^2]}$ bound for $\E\normi{\ch_k}^2$.}

\begin{lemma}
	\label{lem:lya}
    Let the conditions in Theorem \ref{thm:total1} hold.
	Denote the constants as 
	{$p_1:= L\theta^2\sigma^2/(2n\mu^2)$, $p_2 := 42\theta^3L^2\sigma^2/[\mu^3(1-\tlambda_2)]$, $p_3:= 200\theta^3 L^2\sumn\normi{\nabla f_i(x^*)}^2/[n\mu^3(1-\tlambda_2)^3]$, and $p_4:= {150n L\theta^2}/[\mu^2(\theta - 9)]$.
	}

	Then,
	{
	\begin{equation}
		\label{eq:cLv_scvx1}
		\begin{aligned}
			\E\cL_k^v &\leq \frac{4}{\theta-9}\brk{ \frac{p_1}{k + m} + \frac{p_2}{(k + m)^2}+ \frac{p_3}{(k + m)^4}}\\
			&\quad + \prt{\frac{m}{k + m}}^{\frac{\theta}{2}}\E\cL_0^v.
		\end{aligned}
	\end{equation}
	}
	In addition,
	{
	\begin{equation}
		\label{eq:cons1}
		\begin{aligned}
			&\E\brk{\norm{\ch_k}^2}\leq \E\cL^b_{k} \leq \prt{\frac{5+ \tlambda_2}{6}}^k\E\cL_0^b \\
			& + \frac{72\theta^2n\sigma^2}{\mu^2(1-\tlambda_2)(k + m)^2}+ \frac{12p_4p_1}{(1-\tlambda_2)(k + m)^3}\\
			&+  \frac{12p_4p_2}{(1-\tlambda_2)(k + m)^4}+ \frac{432\theta^2 \sumn\norm{\nabla f_i(x^*)}^2}{\mu^2(1-\tlambda_2)^3(k + m)^4}\\
			&+ \frac{12p_4 p_3}{(1-\tlambda_2)(k + m)^6}+ \frac{12p_4 m^{\frac{\theta}{2}}\E\cL_0^v}{(1-\tlambda_2)(k + m)^{\theta / 2 + 2}}.
		\end{aligned}
	\end{equation}
	}

\end{lemma}

\begin{proof}
	See Appendix \ref{app:lem_lya}.
\end{proof}

\section{Numerical Experiments}
\label{sec:exp}
In this section, we present the numerical results regarding logistic regression and neural network training. For the compression schemes, we consider Tok-$K$ and Random-$K$ as the biased compressor $\cC_1$ and $\cC_2$ respectively and scaled Random-$K$ as the unbiased compressor $\cC_3$. 
We set $K = \lfloor 5\% \cdot p\rfloor$ by default for compressors $\cC_1$, $\cC_2$, and $\cC_3$, where $p$ is the dimension. We also consider the unbiased $b-$bit quantization $\cC_4$ in \cite{liu2021linear}: $\cC_4(x) =\prt{\norm{x}_{\infty}2^{-(b-1)}\mathrm{sign}(x)}\cdot \left\lfloor\frac{2^{(b-1)}|x|}{\norm{x}_{\infty}} + \mu\right\rfloor,$
where $\cdot$ is the Hadamard product. { The operators} $\mathrm{sign}(\cdot)$, {$\lfloor\cdot\rfloor$,} and $|\cdot|$ are implemented element-wisely. The vector $\mu$ is random and uniformly distributed in $[0,1]^p$. In the following, we choose $b$ such that the bits sent are $5\%$ of the uncompressed schemes. 

The network topologies we consider can be found in { Fig.} \ref{fig:network}. Each node in the exponential network is connected to its $2^0, 2^1, 2^2, \cdots$ hops neighbors. The mixing matrices compliant with these two networks are constructed under the Lazy Metropolis rule \cite{nedic2018network}. 

\begin{figure}[htbp]
	\centering
	\subfloat[Exponential network, $n = 16$.]{\includegraphics[width = 0.24\textwidth]{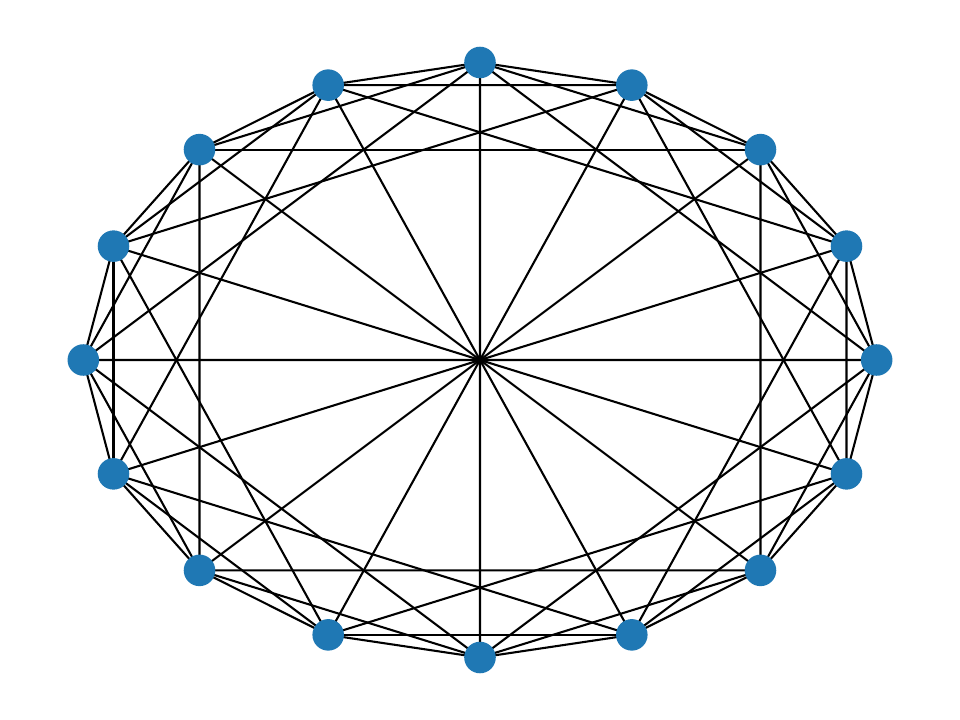}}
	\subfloat[Grid network, $n = 16$.]{\includegraphics[width = 0.24\textwidth]{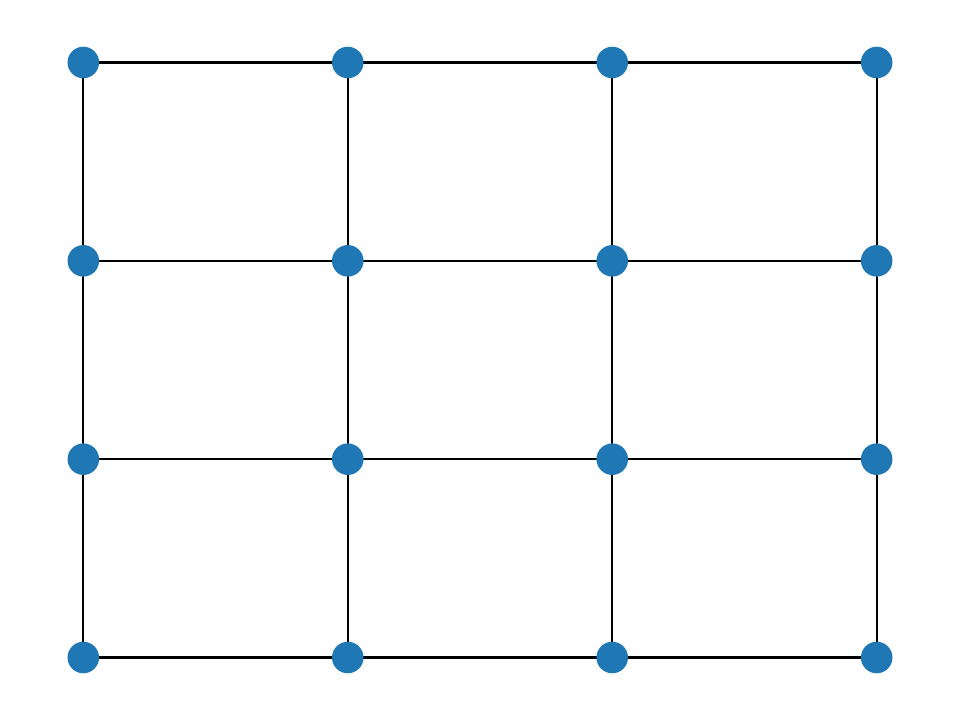}}
	\caption{Illustration of two network topologies.}
	\label{fig:network}
\end{figure}

\subsection{Logistic Regression}
\label{subsec:logistic}

We consider a binary classification problem using logistic regression \eqref{eq:logistic}. Each agent possesses a distinct local dataset $\mathcal{S}_i$ selected from the whole dataset $\mathcal{S}$. The classifier can then be obtained by solving the following convex optimization problem using all the agents' local datasets $\mathcal{S}_i, i=1,2,...,n$:
\begin{subequations}
	\label{eq:logistic}
	\begin{align}
		&\min_{x\in\R^{785}} f(x) = \frac{1}{n}\sum_{i=1}^n f_i(x),\\
		&f_i(x) := \frac{1}{|\mathcal{S}_i|} \sum_{j\in\mathcal{S}_i} \log\left[1 + \exp(-x^{\T}u_jv_j)\right] + \frac{\rho}{2}\norm{x}^2.
	\end{align}
\end{subequations}

For Problem \eqref{eq:logistic}, we first consider the MNIST dataset \cite{mnist}. We set $\rho = 1/5$ and use decreasing stepsize $5 /(k + 100)$ for all the algorithms. The parameters of Algorithm \ref{alg:lead_ds} are chosen as $\alpha = 0.1$ and $\gamma = 0.004$. As shown in Table \ref{tab:kt}, it is sufficient to compare CEDAS with Choco-SGD, which previously enjoys the shortest transient time. The parameter of Choco-SGD is also set as $\gamma = 0.004$. The performance of those uncompressed decentralized methods is presented as the baseline. 

We first illustrate the performance of different algorithms via the residual error $\sumn\E\brki{\normi{x_{i,k} - x^*}^2} / n$ against the number of iterations in { Fig.} \ref{fig:iter_logistic}. { Such a comparison illustrates the asymptotic network independent property of CEDAS.} 
It can be seen that, regarding the asymptotic network independent property of CEDAS, 
the results are consistent with our theoretical finding, that is, it takes more iterations for CEDAS to achieve comparable performance with centralized SGD when the connectivity of the network becomes worse (from an exponential network to a grid network). Moreover, the performance of CEDAS is better than that of Choco-SGD for both graphs (under different compressors). The superiority of CEDAS is more evident in a grid network. 

\begin{figure}[htbp]
	\centering
	\subfloat[Grid network, $n = 100$, $1-\lambda_2 = 0.013$]{\includegraphics[width = 0.24\textwidth]{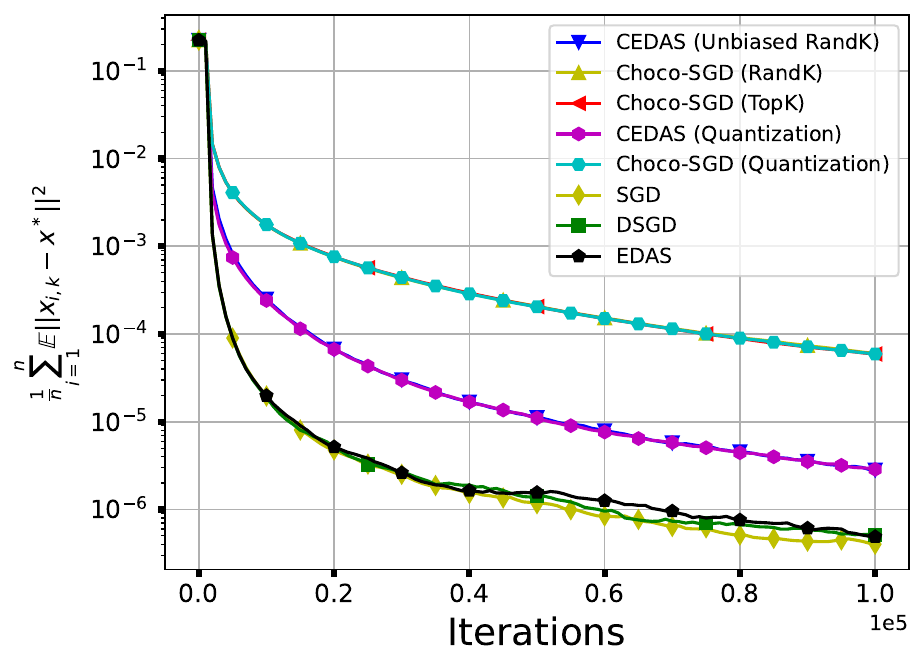}}
	\subfloat[Exponential network, $n = 100$, $1-\lambda_2 = 0.133$.]{\includegraphics[width = 0.24\textwidth]{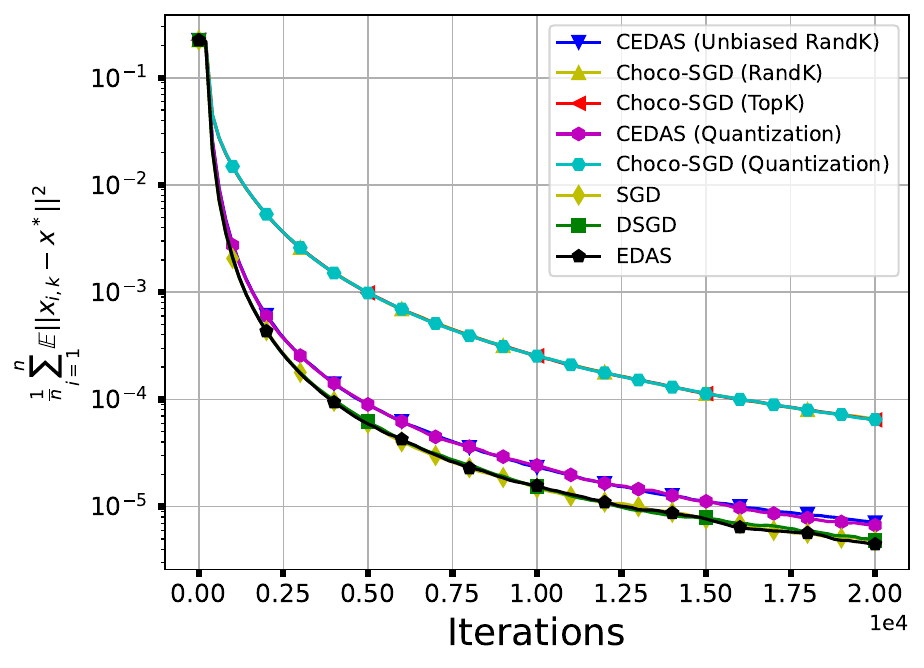}}
	\caption{Residual against the number of iterations. The results are averaged over $10$ repeated runs.}
	\label{fig:iter_logistic}
\end{figure}

Then we check the variation of the residual error while fixing the total transmitted bits of each node in { Fig.} \ref{fig:bits_logistic}. { It aims to demonstrate that communication compression enhances the convergence when the transmitted bits are limited.} 
As shown in { Fig.} \ref{fig:bits_logistic}, the compressed decentralized methods achieve better accuracy compared to their uncompressed counterparts under the same transmitted bits. In particular, CEDAS achieves better performance compared to Choco-SGD. Such a difference becomes more evident when the graph connectivity is worse (in a grid network).

\begin{figure}[htbp]
	\centering
	\subfloat[Grid network, $n = 100$, $1-\lambda_2 = 0.013$]{\includegraphics[width = 0.24\textwidth]{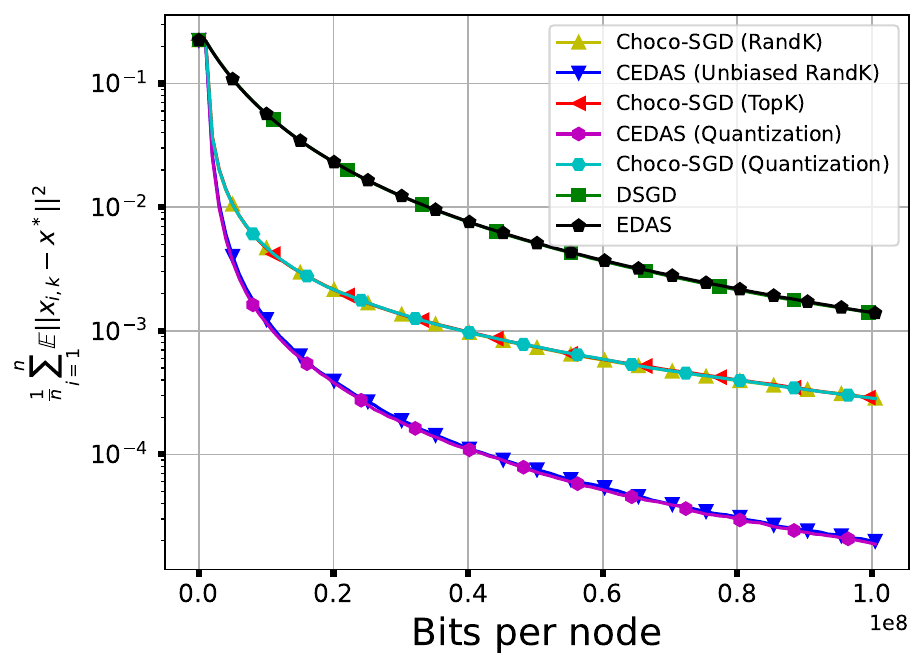}}
	\subfloat[Exponential network, $n = 100$, $1-\lambda_2 = 0.133$.]{\includegraphics[width = 0.24\textwidth]{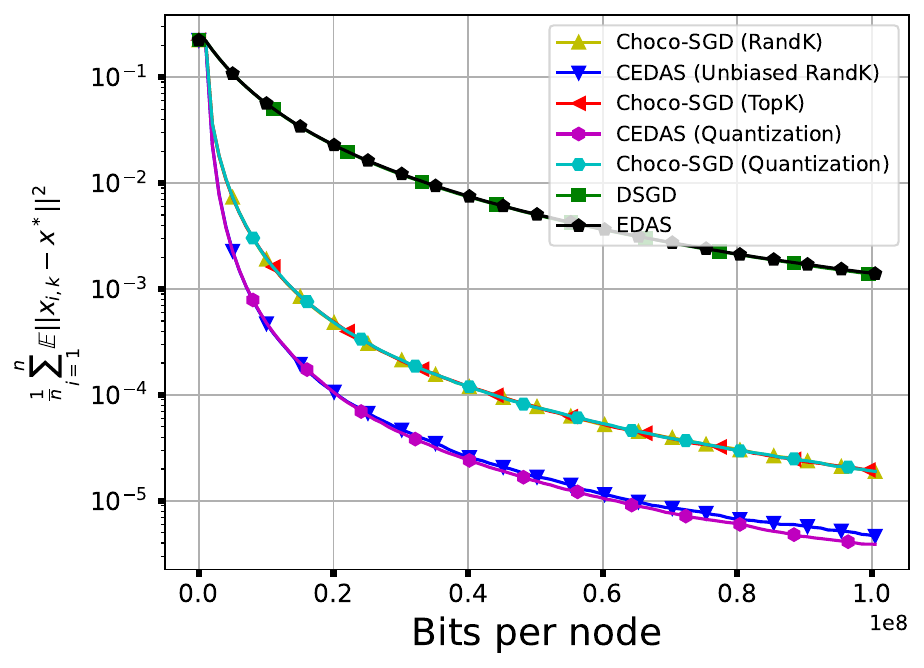}}
	\caption{Residual against the communicated bits. The results are averaged over $5$ repeated runs.}
	\label{fig:bits_logistic}
\end{figure}

\subsection{Neural Network}

\begin{figure}[htbp]
	\centering
	\subfloat[Grid network, $n = 25$, $1-\lambda_2 =  0.054$]{\includegraphics[width = 0.24\textwidth]{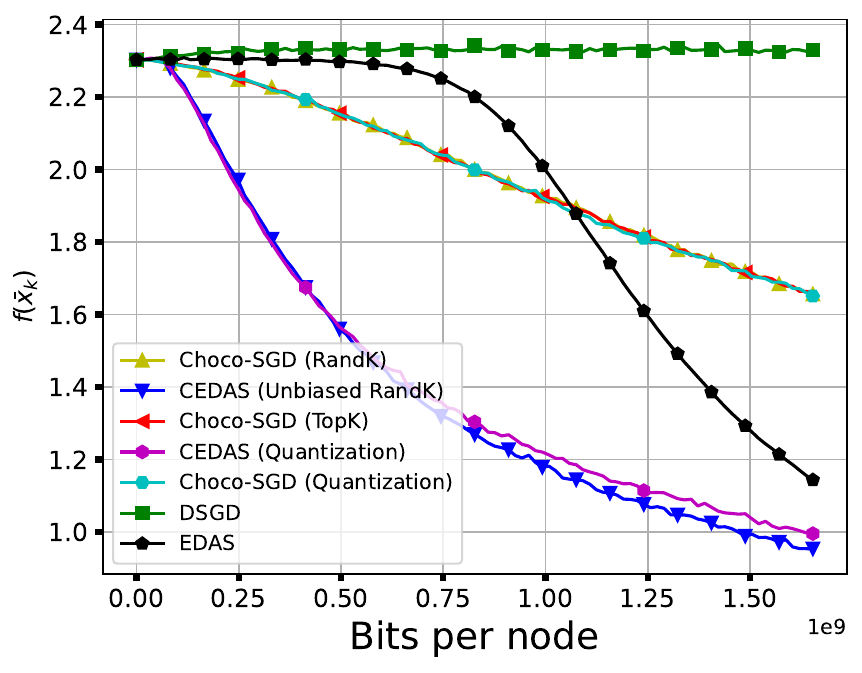}}
	\subfloat[Exponential network, $n = 25$, $1-\lambda_2 = 0.305$.]{\includegraphics[width = 0.24\textwidth]{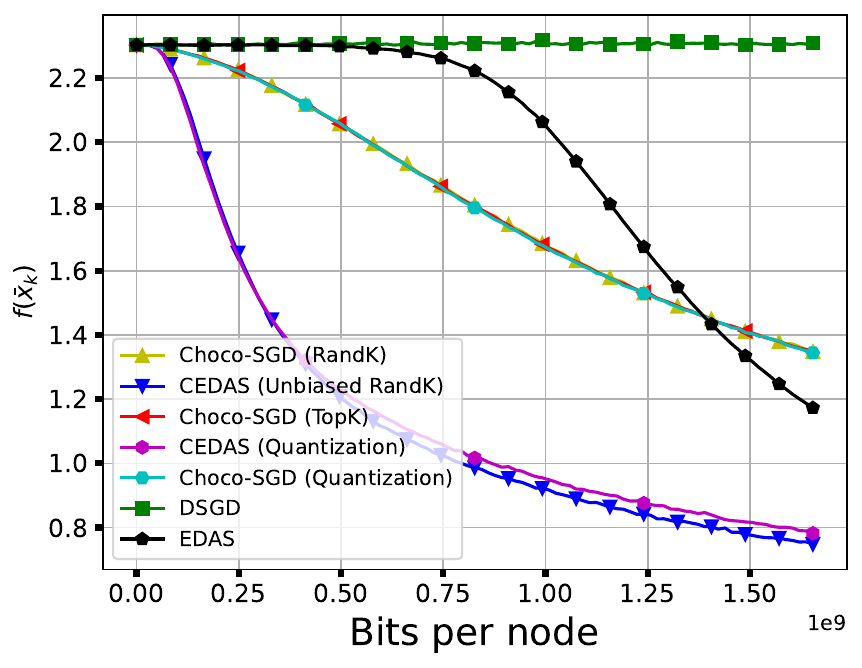}}
	\caption{Loss against communicated bits. The results are averaged over $2$ repeated runs.}
	\label{fig:bits_nn}
\end{figure}
For the nonconvex case, we consider training a neural network with one hidden layer of $64$ neurons for a $10$-class classification problem on the MNIST dataset. We use constant stepsize $\eta = 0.1$ and set the parameter $\alpha=0.1$ and $\gamma = 0.004$. The dimension of the problem is $p=51675$. { Fig.} \ref{fig:bits_nn} illustrates the performance of Choco-SGD and CEDAS with different compressors and their uncompressed counterparts DSGD and EDAS. { Fig. \ref{fig:bits_nn} corroborates our theoretical findings that communication compression benefits the performance of decentralized methods for nonconvex objective functions.} When fixing the total transmitted bits, the compressed decentralized methods are preferable compared to their uncompressed counterparts. Particularly, the proposed CEDAS algorithm enjoys the best performance as illustrated in { Fig.} \ref{fig:bits_nn}.

{
\section{Conclusions}
\label{sec:conc}
This paper focuses on addressing the distributed stochastic optimization problem over networked agents. We investigate a new method termed ``compressed exact diffusion with adaptive stepsizes (CEDAS)'' which shortens the transient time for minimizing smooth objective functions. In particular, to our knowledge, CEDAS enjoys so far the shortest transient time (with respect to the graph specifics) for achieving the convergence rate of centralized SGD among communication compressed algorithms. Experimental results are provided that corroborate the theoretical findings. 
}

\appendices

\section{LEAD Algorithm}
\label{app:lead}
\begin{algorithm}[H]
	\caption{LEAD in Agent's Perspective}
	\label{alg:lead}
	\begin{algorithmic}[1]
		\Require Stepsize $\eta$, parameters $\gamma$ and $\alpha$, and initial values $x_{i,-1}, h_{i,0},z_i,i\in[n]$.
		\For{Agent $i$ in parallel}
		\State $a_{i,0} = z_i - \sum_{j\in \cN_i\cup\crk{i}}w_{ij}z_j$
		\State $(h_w)_{i, 0} = \sum_{j\in \cN_i\cup\crk{i}}w_{ij}h_{j,0}$
		\State Compute $\nabla f_i(x_{i,-1};\xi_{i, -1})$ and $x_{i,0} = x_{i,-1} - \eta \nabla f_i(x_{i,-1};\xi_{i, -1})$
		\EndFor
		\For{$k = 0,1,\cdots, K-1$, agent $i$ in parallel}
		\State Compute $\nabla f_i(x_{i,k};\xi_{i,k})$
		\State $y_{i,k} = x_{i,k} - \eta\nabla f_i(\x_{i,k};\xi_{i,k}) - \eta a_{i,k}$
		\State { Obtain} $(\hat{y}_{i,k}, (\hat{y}_w)_{i,k}, h_{i,k+1}, (h_w)_{i,k + 1})$ by querying COMM $(y_{i,k}, h_{i,k}, (h_w)_{i,k},\alpha)$ 
		\State $a_{i,k + 1} = a_{i,k} + \frac{\gamma}{2\eta}\prt{\hat{y}_{i,k} - (\hat{y}_w)_{i,k}}$
		\State $x_{i,k + 1} = x_{i,k} - \eta\nabla f_i(x_{i,k};\xi_{i,k}) - \eta a_{i, k + 1}$
		\EndFor
		\State Output $x_{i, K}$.
	\end{algorithmic}
\end{algorithm}

\section{Proofs}

\subsection{Proof of Lemma \ref{lem:hk}}
\label{app:lem_hk}
    {
	Recall $\hat{\r}_k = \tV\s_k + \eta_k\nabla F(\1\bar{x}_k^{\T})$. According to \eqref{eq:edas_c} and Lemma \ref{lem:tW}, we have 
    \begin{equation}
		\label{eq:hrk}
		\begin{aligned}
				&\hat{\r}_{k + 1} 
                = \tW\hat{\r}_k + (I-\tW)\x_k + (I-\tW)\eta_k\hat{\g}_k + (I-\tW)E_k\\
                & + \eta_{k + 1}\brk{\nabla F(\1\bar{x}_{k + 1}^{\T}) - \nabla F(\1\bar{x}_k^{\T})} + (\eta_{k + 1} - \eta_k)\nabla F(\1\bar{x}_k^{\T})\\
                & + (I-\tW)\eta_k\brk{\nabla F(\1\bar{x}_k^{\T}) - \nabla F(\x_k)}.
			\end{aligned}
	\end{equation}
	For the recursion of $\x_k$, we have 
	\begin{equation}
		\label{eq:ncvx_xk}
		\begin{aligned}
			&\x_{k + 1} 
			= \tW\x_k -\tW\hat{\r}_k + \tW\eta_k\hat\g_k - (I-\tW)E_k\\
            &\quad +  \tW\eta_k\brk{\nabla F(\1\bar{x}_k^{\T}) - \nabla F(\x_k)}.
		\end{aligned}
	\end{equation}

    From the eigendecomposition of $\tW = Q\tLambda Q^{\T}$, we have $\tW = {\1\1^{\T}}/{n} + Q_1\tLambda_1Q_1^{\T}$ and $I-\tW = Q_1(I-\tLambda_1)Q_1^{\T}$.
	Note that $Q_1^{\T}\1 = \0$ and $Q_1^{\T}Q_1 = I_{n-1}$. Multiplying $(I-\tLambda_1)^{-1/2} Q_1^{\T}$ and $Q_1^{\T}$ to \eqref{eq:hrk} and \eqref{eq:ncvx_xk} { respectively} leads to 

    }

	\small
	\begin{equation}
		\label{eq:Qxk}
		\begin{aligned}
			&{\hat{\bs{e}}_{k + 1}= B\hat{\bs{e}}_k}
			+ \begin{pmatrix}
				\tLambda_1 Q_1^{\T}\eta_k\hat{\g}_k - (I-\tLambda_1)Q_1^{\T}E_k\\
				\sqrt{I-\tLambda_1}Q_1^{\T}\eta_k\hat{\g}_k + \sqrt{I-\tLambda_1}Q_1^{\T}E_k
			\end{pmatrix}\\
			&+\begin{pmatrix}
				\0\\
				(I-\tLambda_1)^{-\frac{1}{2}} Q_1^{\T}\eta_{{ k + 1}} \brk{\nabla F(\1\bar{x}_{k + 1}^{\T})-\nabla F(\1\bar{x}_k^{\T})}
			\end{pmatrix}\\
            & { +
                \begin{pmatrix}
                    \0\\
                    (I-\tLambda_1)^{-\frac{1}{2}}Q_1^{\T}(\eta_{k + 1} - \eta_k)\nabla F(\1\bar{x}_k^{\T})
                \end{pmatrix}}\\
			& + \begin{pmatrix}
				\tLambda_1 Q_1^{\T}\eta_{ k} \brk{\nabla F(\1\bar{x}_k^{\T}) - \nabla F(\x_k)}\\
				{\sqrt{I-\tLambda_1}Q_1^{\T}\eta_k\brk{\nabla F(\1\bar{x}_k^{\T}) - \nabla F(\x_k)}}
			\end{pmatrix},
		\end{aligned}
	\end{equation}\normalsize
    { where 
	\small
    \begin{align*}
        \bs{e}_k &:= \begin{pmatrix}
            Q_1^{\T}\x_k\\
            (I-\tLambda_1)^{-\frac{1}{2}}Q_1^{\T}\hat{\r}_k
        \end{pmatrix},     
        B:= \begin{pmatrix}
            \tLambda_1 & - \tLambda_1 \sqrt{I-\tLambda_1} \\
            \sqrt{I-\tLambda_1} & \tLambda_1
        \end{pmatrix}.
    \end{align*}\normalsize
    }

	Noting the structure of ${ B}$ in \eqref{eq:Qxk}, we conclude according to{\cite[Lemma 4]{huang2021improving}} that $B$ has a decomposition $B= JP_1 J^{-1}$ with
	\begin{equation}
		\label{eq:js}
		\begin{aligned}
			J&
            :=\begin{pmatrix}
				J_{R,u}\\
				J_{R, l}
			\end{pmatrix},\;
			J^{-1}&:=\begin{pmatrix}
				J_{L,l} & J_{L,r}
			\end{pmatrix},
		\end{aligned}
	\end{equation}
	for some $J_{R,u}, J_{R,l}\in\R^{(n-1)\times 2(n-1)}$ and $J_{L,l}, J_{L,r}\in\R^{2(n-1)\times (n-1)}$. 
    Moreover, we have { from \cite[Lemma 4]{huang2021improving} that} $\normi{J}_2\leq \sqrt{2}$, $\normi{J^{-1}}_2\leq 1/\sqrt{2\tlambda_n}$,{$\normi{P_1}^2\leq \tlambda_2$, and $\normi{P_1-I}^2=\gamma(1-\lambda_n)/2$.}

	The recursion of $\ch_k$ is then obtained by multiplying $J^{-1}$ to \eqref{eq:Qxk} { and noting that $J_{L,l}\tLambda_1 + J_{L,r}\sqrt{I-\tLambda_1} = P_1J_{L,l}$}.
	In addition, $Q_1^{\T}\x_k = J_{R,u}\ch_k$, and  $(I-\tLambda_1)^{-1/2}Q_1^{\T}\hat{\r}_k = J_{R,l}\ch_k.$ 


    
    We next explore the relation between $\normi{\ch_k}^2$ and the consensus error $\normi{\x_k - \1\bar{x}_k^{\T}}^2$. Note $Q_1Q_1^{\T} = I - \1\1^{\T}/n$ and $Q_1^{\T}\x_k = J_{R,u}\ch_k$, we have $\x_k - \1\bar{x}_k^{\T} = Q_1J_{R,u}\ch_k.$


\subsection{Proof of Lemma \ref{lem:descent0}}
\label{app:ncvx_lem_descent0}
{ The objective function $f$ is $L$-smooth according to Assumption \ref{ass:smooth}. Then by descent Lemma and Assumption \ref{ass:sgrad}, we have 
    \begin{equation}
        \label{eq:descent_start}
        \begin{aligned}
            &\condE{f(\bar{x}_{k + 1})}{\cG_k} \leq f(\bar{x}_k) - \inproi{\nabla f(\bar{x}_k), \frac{\eta_k}{n}\sumn f_i(x_{i,k})}\\
			&\quad + \frac{L\eta_k^2}{2}\condE{\norm{\frac{1}{n}\sumn g_{i,k}}^2}{\cG_k}\\
            &\leq f(\bar{x}_k) - \frac{\eta_k}{2}\norm{\nabla f(\bar{x}_k)}^2+ \frac{L\eta_k^2\sigma^2}{2n}\\ 
            &\quad +\frac{\eta_k}{2}\norm{\nabla f(\bar{x}_k) - \frac{1}{n}\sumn\nabla f_i(x_{i,k})}^2,
        \end{aligned}        
    \end{equation}
    where in the last inequality we let $\eta_k\leq 1/(2L)$, and invoke Assumption \ref{ass:sgrad} as well as $2\inproi{a,b} = \normi{a}^2 + \normi{b}^2 - \normi{a - b}^2,\ \forall a,b\in\R^p.$
    Notice that from Assumption \ref{ass:smooth} and Lemma \ref{lem:hk}, we have $\normi{\nabla f(\bar{x}_k) - \frac{1}{n}\sumn\nabla f_i(x_{i,k})}^2 \leq 
	\frac{2L^2}{n}\normi{\ch_k}^2$.
    Substituting such a bound into \eqref{eq:descent_start}, and taking full expectation yield the desired result. 
}

\subsection{Proof of Lemma \ref{lem:ncvx_cons0}}
\label{app:ncvx_lem_cons0}
	{ 
	Recall that $\normi{P_1}^2 = \tlambda_2$, $\normi{P_1 - I}^2\leq \gamma(1-\lambda_2)/2$, and $\normi{(I-\tLambda_1)^{-1/2}}^2\leq 1/(1-\tlambda_2)$. According to Assumptions \ref{ass:sgrad}, \ref{ass:ub_op}, \ref{ass:smooth}, tower property, and Young's inequality for some $q>0$, we have for $\gamma\leq 1/2$ that 
	\begin{equation}
		\label{eq:ncvx_cons_s1}
		\begin{aligned}
			&\condE{\norm{\ch_{k + 1}}^2}{\cG_k}\leq (1 + q)\tlambda_2\norm{\ch_k}^2 + 3n\eta_k^2\sigma^2\\
			&\quad + 2\gamma\condE{\norm{E_k}^2}{\cG_k} + \frac{12(1 + q)\eta_k^2L^2}{q}\norm{\ch_k}^2\\
			&\quad + \frac{6(1 + q)(\eta_{k + 1} - \eta_k)^2}{q(1-\tlambda_2)}\norm{\nabla F(\1\bar{x}_k^{\T})}^2\\
			&\quad + \frac{(6+7q)n\eta_k^2\eta_{k + 1}^2 L^2}{q(1-\tlambda_2)}\condE{\norm{\frac{1}{n}\sumn g_{i,k}}^2}{\cG_k}.
		\end{aligned}
	\end{equation}

	Noting that $\cG_k\subset\cF_k$, we have from tower property and Assumption \ref{ass:ub_op} that 
	\begin{equation}
		\label{eq:Ek}
		\begin{aligned}
			\condE{\norm{E_k}^2}{\cG_k} &= \condE{\condE{\norm{E_k}^2}{\cF_k}}{\cG_k}\\
			&\leq C\condE{\norm{\y_k - H_k}^2}{\cG_k}.
		\end{aligned}
	\end{equation}

	For the last term in \eqref{eq:ncvx_cons_s1}, we consider the split that $g_{i,k} = g_{i,k} - \nabla f_i(x_{i,k}) + \nabla f_i(x_{i,k}) - \nabla f(\bar{x}_k) + \nabla f(\bar{x}_k)$. Then, 
	\begin{equation}
		\label{eq:gk_mean}
		\begin{aligned}
			&\frac{1}{3}\condE{\norm{\frac{1}{n}\sumn g_{i,k}}^2}{\cG_k}\leq \frac{\sigma^2}{n} + \frac{2L^2}{n}\norm{\ch_k}^2 + \norm{\nabla f(\bar{x}_k)}^2.
		\end{aligned}
	\end{equation}

	Substituting \eqref{eq:Ek} and \eqref{eq:gk_mean} into \eqref{eq:ncvx_cons_s1} and letting $q = (1-\tlambda_2)/(2\tlambda_2)$ yield
	\begin{equation}
		\label{eq:ncvx_cons_s2}
		\begin{aligned}
			&\condE{\norm{\ch_{k + 1}}^2}{\cG_k} \leq 2\gamma C
			\condE{\norm{\y_k - H_k}^2}{\cG_k} \\
			&\quad + \brk{\frac{1+\tlambda_2}{2} + \frac{24\eta_k^2 L^2}{1-\tlambda_2}\prt{1 + \frac{3\eta_{k+1}^2 L^2}{1-\tlambda_2}}}\norm{\ch_k}^2\\
			&\quad + \frac{36n\eta_k^4L^2}{(1-\tlambda_2)^2}\norm{\nabla f(\bar{x}_k)}^2 + 2n\eta_k^2\sigma^2\prt{\frac{3}{2} + \frac{18\eta_{k+1}^2L^2}{n(1-\tlambda_2)^2}}\\
			&\quad + \frac{12(\eta_{k + 1} - \eta_k)^2\norm{\nabla F(\1\bar{x}_k^{\T})}^2}{(1-\tlambda_2)^2}.
		\end{aligned}
	\end{equation}
	Let $\eta_k$ satisfy $\eta_k\leq (1-\tlambda_2)/(16L)$. Noting the decreasing stepsize policy and taking the full expectation on \eqref{eq:ncvx_cons_s2} yield the desired result. 
	}


\subsection{Proof of Lemma \ref{lem:ncvx_yh}}
\label{app:ncvx_lem_yh0}

	{ We start by considering bounding the terms $\normi{(I-\tW)\y_k}^2$ and $\normi{\hat{\r}_k}^2$. Recall from the derivation in Lemma \ref{lem:hk} that $Q_1^{\T}\hat{\r}_k=\sqrt{I-\tLambda_1}J_{R,l}\ch_k$ and $Q_1Q_1^{\T} = I-\1\1^{\T}/n$. We have $\hat{\r}_k 
	= Q_1\sqrt{I-\tLambda_1}J_{R,l}\ch_k  + \1\1^{\T}\eta_k/n\nabla F(\1\bar{x}_k^{\T})$.
	Therefore,
	\begin{equation}
		\label{eq:hrk_up}
		\begin{aligned}
			\norm{\hat{\r}_k}^2\leq 4\gamma\norm{\ch_k}^2 + 2\eta_k^2 n \norm{\nabla f(\bar{x}_k)}^2.
 		\end{aligned}
	\end{equation}

	Note from \eqref{eq:lead_ds} that $\prti{I-\tW}\y_k = Q_1(I-\tLambda_1)\crki{Q_1^{\T}\x_k + Q_1^{\T}\eta_k\hat{\g}_k - Q_1^{\T}\hat{\r}_k+Q_1^{\T}\eta_k\brki{\nabla F(\1\bar{x}_k^{\T}) - \nabla F(\x_k)}}$ and $Q_1^{\T}\x_k - Q_1^{\T}\hat{\r}_k = (J_{R,u}-\sqrt{I-\tLambda_1} J_{R,l})\ch_k$. Then, 
	\begin{equation}
		\label{eq:Iy}
		\begin{aligned}
			&\condE{\norm{(I-\tW)\y_k}^2}{\cG_k}\leq 4\gamma^2\prt{ 1+\eta_k^2L^2}\norm{\ch_k}^2\\
			&\quad +  \eta_k^2\gamma^2n\sigma^2 + 4\gamma^2\eta_k^2n\norm{\nabla f(\bar{x}_k)}^2.
		\end{aligned}
	\end{equation}

	From $H_{k + 1} = H_k + \alpha\cC(\y_k - H_k)$ and \eqref{eq:lead_ds}, we have 
	\begin{equation}
		\label{eq:ykhk}
		\begin{aligned}
			&\y_{k + 1} - H_{k + 1} = \cA_1 - 2(I-\tW)\y_k - \hat{\r}_k\\
			&\quad + \eta_k\nabla F(\1\bar{x}_k^{\T}) - \eta_{k + 1}\nabla F(\x_{k + 1})+ \eta_{k + 1}\hat{\g}_{k + 1},
		\end{aligned}
	\end{equation}
	where $\cA_1:= (1-\alpha )(\y_k - H_k) - \alpha E_k - 2(I-\tW)E_k$. 
	
	In light of Assumption \ref{ass:sgrad} and $\eta_{k + 1}\leq \eta_k$, we have from tower property and Young's inequality for $q>0$ that 
	\begin{equation}
		\label{eq:yh_s1}
		\begin{aligned}
			&\condE{\norm{\y_{k + 1} - H_{k + 1}}^2}{\cG_k}\leq (1 + q)\condE{\norm{\cA_1}^2}{\cG_k}  + 2\eta_{k}^2n\sigma^2\\
			&+ \frac{3(1+q)}{q}\crk{\condE{\norm{2(I-\tW)\y_k}^2}{\cG_k}  + \norm{\hat{\r}_k}^2}\\
			&+ \frac{(3+4q)}{q}\condE{\norm{\eta_k\nabla F(\1\bar{x}_k^{\T}) - \eta_{k + 1} \nabla F(\x_{k + 1})}^2}{\cG_k}.
		\end{aligned}
	\end{equation}

	We next consider bounding the last term in \eqref{eq:yh_s1}. Note that $\eta_k\nabla F(\1\bar{x}_k^{\T}) - \eta_{k + 1} \nabla F(\x_{k + 1}) = \eta_k\brki{\nabla F(\1\bar{x}_k^{\T}) - \nabla F(\x_k)} + (\eta_k - \eta_{k + 1})\nabla F(\x_k) + \eta_{k + 1}\brki{\nabla F(\x_k) - \nabla F(\x_{k + 1})}$, then 
	\begin{equation}
		\label{eq:yh_nf_diff}
		\begin{aligned}
			&\condE{\norm{\eta_k\nabla F(\1\bar{x}_k^{\T}) - \eta_{k + 1} \nabla F(\x_{k + 1})}^2}{\cG_k}\\
			&\leq 6\eta_k^2 L^2\norm{\ch_k}^2 + 3\eta_{k + 1}^2 L^2\condE{\norm{\x_k - \x_{k + 1}}^2}{\cG_k}\\
			&\quad + 3(\eta_k - \eta_{k + 1})^2\norm{\nabla F(\x_k)}^2.
		\end{aligned}
	\end{equation}

	Note that $\x_k - \x_{k + 1} = -\eta_k\hat{\g}_k + \hat{\r}_k + \eta_k\brki{\nabla F(\x_k) - \nabla F(\1\bar{x}_k^{\T})} + (I-\tW)\prti{E_k + \y_k}$. We have 
	\begin{equation}
		\label{eq:x_diff}
		\begin{aligned}
			&\condE{\norm{\x_k - \x_{k + 1}}^2}{\cG_k}\leq n\eta_k^2\sigma^2 + \gamma^2 C\condE{\norm{\y_k - H_k}^2}{\cG_k}\\
			&\quad + 3\norm{\hat{\r}_k}^2 + 6\eta_k^2 L^2\norm{\ch_k}^2 + 3\condE{\norm{(I-\tW)\y_k}^2}{\cG_k}.
		\end{aligned}
	\end{equation}

	We then consider bounding the term $\condEi{\normi{\cA_1}^2}{\cG_k}$. In light of Assumption \ref{ass:ub_op}, we have 
	\begin{equation}
		\label{eq:yh_s2}
		\begin{aligned}
			&\condE{\norm{\cA_1}^2}{\cG_k} \leq 2\alpha \condE{\inpro{E_k, 2(I-\tW)E_k}}{\cG_k}\\
			&+ \brk{(1-\alpha )^2 + \alpha^2 C + 4\gamma^2 C}\condE{\norm{\y_k - H_k}^2}{\cG_k}.
		\end{aligned}
	\end{equation}

    By Cauchy-Schwarz inequality, we have 
    \begin{align*}
        &\condE{\inpro{E_k, (I-\tW)E_k}}{\cF_k}\\
		&\leq \sqrt{\condE{\norm{E_k}^2}{\cF_k}\condE{\norm{(I-\tW)E_k}^2}{\cF_k}}\\
		&\leq \gamma C\norm{\y_k - H_k}^2.
    \end{align*}

	We next determine $q$ in \eqref{eq:yh_s1} such that 
	\begin{equation}
		\label{eq:yh_q}
		\begin{aligned}
			(1 + q)\brk{(1-\alpha )^2 + \alpha^2 C + 4\gamma^2 C + 4\alpha\gamma C} \leq 1 - \frac{\alpha}{2}.
		\end{aligned}
	\end{equation}
	
	Letting $\gamma\leq \alpha/2$ yields $(1-\alpha )^2 + \alpha^2 C + 4\gamma^2 C + 4\alpha\gamma C\leq (1-\alpha)^2+4\alpha^2 C$. It then suffices to let $\alpha\leq 1/[4(C+1)]$ and choose $q = \alpha/[2(1-\alpha)]$ for \eqref{eq:yh_q} to hold. Combining \eqref{eq:hrk_up}-\eqref{eq:yh_q} leads to 

	\begin{equation}
		\label{eq:yh_s3}
		\begin{aligned}
			&\condE{\norm{\y_{k + 1} - H_{k + 1}}^2}{\cG_k} 
			\leq \frac{6(1 + 9\eta_k^2 L^2)}{\alpha }\norm{\hat{\r}_k}^2 \\
			&\quad + \prt{1 - \frac{\alpha }{2} + \frac{18\eta_k^2 L^2\gamma^2 C}{\alpha }}\condE{\norm{\y_k - H_k}^2}{\cG_k}\\
			&\quad + \frac{6(4 + 9\eta_k^2 L^2)}{\alpha }\condE{\norm{(I-\tW)\y_k}^2}{\cG_k}\\
			&\quad + \frac{18(\eta_k -\eta_{k + 1})^2}{\alpha }\norm{\nabla F(\x_k)}^2 + \prt{2 + \frac{18\eta_k^2 L^2}{\alpha }}n\eta_k^2 \sigma^2\\
			&\quad + \frac{36\eta_k^2 L^2(1 + 3\eta_k^2 L^2)}{\alpha }\norm{\ch_k}^2\\
			&\leq \prt{1 - \frac{\alpha }{3}}\condE{\norm{\y_k - H_k}^2}{\cG_k} \\
			&\quad + \frac{24\gamma[(1+9\eta_k^2L^2)+\gamma(4+9\eta_k^2L^2) + 2\eta_k^2L^2/\gamma]}{\alpha}\norm{\ch_k}^2\\
			&\quad + \frac{12n\eta_k^2[(1+9\eta_k^2 L^2) + 2\gamma^2(4+9\eta_k^2 L^2)]}{\alpha}\norm{\nabla f(\bar{x}_k)}^2\\
			&\quad + \prt{2 + \frac{18\eta_k^2 L^2 + 6(1 + 9\eta_k^2 L^2)\gamma^2}{\alpha }}\eta_k^2n\sigma^2\\
			&\quad + \frac{18(\eta_k - \eta_{k + 1})^2}{\alpha }\norm{\nabla F(\x_k)}^2.
		\end{aligned}
	\end{equation}
	Letting $\eta_k\leq \gamma/(3L)$, $\gamma\leq \min\crki{\alpha /2, 1/2}$, and taking the full expectation in \eqref{eq:yh_s3} yield the desired result.

	}

\subsection{Proof of Lemma \ref{lem:lya_ncvx}}
\label{app:lya_ncvx}
{ We set $\eta_k = \eta$ in the following. According to Lemmas \ref{lem:descent0}-\ref{lem:ncvx_yh}, we have 
    \begin{equation}
		\label{eq:lya_ncvx_s1}
		\begin{aligned}
			&\E\cL_{k + 1}^v \leq \prt{\E f(\bar{x}_k) - \finf} +\prt{4v_1+ 6v_2}n\eta^3\sigma^2\\
			& + \prt{\frac{ L^2}{n} + \frac{3+\tlambda_2}{4}v_1+ \frac{124\gamma v_2}{\alpha} }\eta\E\brk{\norm{\ch_k}^2} + \frac{L\eta^2\sigma^2}{2n}\\
			&+\brk{2 v_1\gamma C + \prt{1 - \frac{\alpha }{3}}v_2}\eta\E\brk{\norm{\y_k - H_k}^2}\\
			& - \frac{\eta}{2}\prt{1 - \frac{72n\eta^4 L^2 v_1}{(1-\tlambda_2)^2} - \frac{108\eta^2 n v_2}{\alpha}}\E\brk{\norm{\nabla f(\bar{x}_k)}^2}.
		    \end{aligned}
	\end{equation}
    
	To derive the recursion for $\E\cL_k^v$, it suffices to have 
	\begin{subequations}
		\begin{align}
			&\frac{L^2}{n}  + \frac{124\gamma v_2}{\alpha} \leq \frac{1-\tlambda_2}{4}v_1,\label{eq:v1}\\
			& 2 v_1\gamma C \leq \frac{\alpha  v_2}{3}.\label{eq:v2}
		\end{align}
	\end{subequations}

	From \eqref{eq:v2}, it suffices to choose $v_2 = 8\gamma C v_1/\alpha$. Substituting such a $v_2$ into \eqref{eq:v1} yields
	$L^2/n\leq [(1-\tlambda_2)/4 - 992\gamma^2 C/\alpha^2]v_1$.
	Letting $\gamma \leq \alpha^2(1-\lambda_2)/(15872C)$, it suffices to choose $v_1 = 8L^2 / [n(1-\tlambda_2)]$ for such an inequality to hold. Therefore, choosing $v_2 = 64\gamma C L^2/[n\alpha (1-\tlambda_2)]$ yields the recursion for $\E\cL_k$:
	\begin{equation}
		\label{eq:lya_cvx_s2}
		\begin{aligned}
			&\E\cL_{k + 1}^v \leq \E \cL_k^v + \frac{L\eta^2 \sigma^2}{2n} + \frac{32\eta^3 L^2\sigma^2 (1 + 12\gamma C / \alpha )}{1-\tlambda_2}\\
			&- \frac{\eta}{2}\brk{1 - \prt{\frac{\eta^2 L^2}{(1-\tlambda_2)^2} + \frac{12\gamma C}{\alpha^2}}\frac{576\eta^2 L^2}{1-\tlambda_2} }\E\brk{\norm{\nabla f(\bar{x}_k)}^2}.
		\end{aligned}
	\end{equation} 

	Letting $\eta\leq \gamma (1-\lambda_2)/(16L)$ yields the desired result.

}

\subsection{Proof of Theorem \ref{thm:KT}}
\label{app:thm_KT}
	Based on Theorem \ref{thm:total1}, we have 
	\begin{align*}
		&\frac{1}{n}\sumn \E\brk{\norm{x_{i,k} - x^*}^2} \leq
		\frac{1}{nk}\left[\order{1} + \prt{\frac{5+\tlambda_2}{6}}^kk\E\cL_0^b\right.\\
		&\left. + \order{\frac{n}{1-\tlambda_2}}\frac{1}{k}+ \order{\frac{1}{1-\tlambda_2}} \frac{1}{k^2} + \order{\frac{n}{(1-\tlambda_2)^2}}\frac{1}{k^3} \right.\\
		&\left.+ \order{\frac{{ \sumn \norm{\nabla f_i(x^*)}^2}}{(1-\tlambda_2)^3}}\frac{1}{k^{ 3}} + \order{{nm^{{\frac{\theta}{2}}} \E\cL_0^{ v}}} \frac{1}{k^{{\frac{\theta}{2} - 1}}}\right].
	\end{align*}

	{ According to \eqref{eq:initial}, we have $n\E\cL_0^v = \orderi{\normi{\x_{-1}}^2 + \normi{H_0}^2 + (1-\tlambda_2)\sumn\normi{\nabla f_i(x^*)}^2}$ and $\E\cL_0^b = \orderi{\normi{\x_{-1}}^2 + (1-\tlambda_2)^2\sumn\normi{\nabla f_i(x^*)}^2 + \norm{H_0}^2 }$.
	}
	Given the definition of $K_T^{(scvx)}$ in \eqref{def:transient}, we obtain the desired transient time $K_T^{(scvx)}$. 

{
\subsection{Proof of Lemma \ref{lem:metric}}
\label{app:lem_metric}
	In light of Assumption \ref{ass:fi}, we have from \cite{nesterov2003introductory} that $f(y) - f(x) \geq \inpro{\nabla f(x), y-x} + \frac{\mu}{2}\norm{y-x}^2,\forall x, y.$
	Setting $y = \bar{x}_k$ and $x = x^*$ in the aforementioned inequality leads to $\normi{\bar{x}_k - x^*}^2\leq 2\brki{f(\bar{x}_k) - f^*}/\mu$.
	We also have $\sumn\normi{x_{i,k} - x^*}^2/n = \normi{\bar{x}_k - x^*}^2 + \sumn\normi{x_{i,k} - \bar{x}_k}^2/n.$
	Combining the above tow inequalities yields the desired result.
}

{
\subsection{Proof of Lemma \ref{lem:cLv_scvx}}
\label{app:lem_cLv_scvx}
Due to Assumption \ref{ass:fi}, we have $\finf = f^*$ and $\normi{\nabla f(x)}^2\geq 2\mu\brki{f(x) - f^*}$ for any $x\in\R^p$ \cite{nesterov2003introductory}. Substituting such a result into \eqref{eq:descent0} yields 
	\begin{equation}
		\label{eq:f_scvx}
		\begin{aligned}
			\E f(\bar{x}_{k + 1}) - f^* &\leq \prt{1 - \eta_k\mu}\brk{\E f(\bar{x}_k) - f^*}\\
			&\quad + \frac{\eta_{ k} L^2}{n}\E\brk{\norm{\ch_k}^2} + \frac{L\sigma^2\eta_{ k}^2}{2n}.
		\end{aligned}
	\end{equation}

	It suffices to consider bounding the terms $\normi{\nabla f(\bar{x}_k)}^2$, $\normi{\nabla F(\1\bar{x}_k^{\T})}^2$, and $\normi{\nabla F(\x_k)}^2$ to derive the enhanced recursion for $\E\cL_k^v$. 
	According to Assumption \ref{ass:smooth}, we have $\normi{\nabla f(x)}^2\leq 2L\brki{f(x) - f^*}$ from \cite{nesterov2003introductory}. For the remaining two terms, we have 
	\begin{equation}
		\label{eq:nF_x}
		\begin{aligned}
			&\norm{\nabla F(\1\bar{x}_k^{\T})}^2\leq 2L^2n\norm{\bar{x}_k - x^*}^2 + 2\sumn\norm{\nabla f_i(x^*)}^2\\
			&\leq \frac{4nL^2}{\mu}\brk{f(\bar{x}_k) - f^*}  + 2\sumn\norm{\nabla f_i(x^*)}^2.
		\end{aligned}
	\end{equation}
	
	Similarly, we have $\normi{\nabla F(\x_k)}^2
	\leq 4nL^2/\mu\brki{f(\bar{x}_k) - f^*} + 4L^2\normi{\ch_k}^2 + 2\sumn\normi{\nabla f_i(x^*)}^2.$
	Therefore, we can rewrite the recursions in Lemmas \ref{lem:ncvx_cons0} and \ref{lem:ncvx_yh} as follows.
	\begin{equation}
		\label{eq:scvx_cons}
		\begin{aligned}
	        &\E\brk{\norm{\ch_{k + 1}}^2} \leq \frac{3 + \tlambda_2}{4}\E\brk{\norm{\ch_k}^2} + 2\gamma C\E\brk{\norm{\y_k - H_k}^2}\\
			& + \frac{24 (\eta_{k + 1} - \eta_k)^2\sumn\norm{\nabla f_i(x^*)}^2}{(1-\tlambda_2)^2} + 4n\eta_k^2\sigma^2\\
			& + \brk{ \frac{2(\eta_{k + 1} - \eta_k)^2}{\mu} + 3\eta_k^4 L}\frac{24 n L^2}{(1-\tlambda_2)^2}\E\brk{f(\bar{x}_k) - f^*}.
	    \end{aligned}
	\end{equation}

	\begin{equation}
		\label{eq:scvx_yh}
		\begin{aligned}
			&\E\brk{\norm{\y_{k+1} - H_{k+1}}^2}\leq \prt{1 - \frac{\alpha }{3}}\E\brk{\norm{\y_k - H_k}^2}\\
			&\quad + \frac{4\brk{31\gamma + 18 L^2(\eta_k - \eta_{k + 1})^2}}{\alpha}\E\brk{\norm{\ch_k}^2}\\
			&\quad +6\eta_k^2 n \sigma^2 + \frac{36(\eta_k - \eta_{k + 1})^2\sumn\norm{\nabla f_i(x^*)}^2}{\alpha }\\
			&\quad + \brk{9\eta_k^2 \mu + 6L(\eta_k - \eta_{k + 1})^2} \frac{12n L}{\alpha \mu}\E\brk{f(\bar{x}_k) - f^*}.
		\end{aligned}
	\end{equation}

	Let $\eta_k$ satisfy 
	\begin{equation}
		\label{eq:etak_s1}
		\begin{aligned}
			&\crk{\brk{ \frac{2(\eta_{k + 1} - \eta_k)^2}{\mu} + 3\eta_k^4 L}\frac{2 n L^2}{(1-\tlambda_2)^2}\right.\\
			&\left. + \brk{9\eta_k^2 \mu  + 6L(\eta_k - \eta_{k + 1})^2} \frac{8\gamma Cn L}{\alpha^2 \mu} }\frac{96\eta_k L^2}{n(1-\tlambda_2)}\leq \frac{\eta_k\mu}{2}.
		\end{aligned}
	\end{equation}

	Note that $(\eta_k - \eta_{k + 1})^2 \leq \theta^2/[\mu^2(k + m)^4]$ due to the stepsize policy $\eta_k = \theta/[\mu(k + m)], k \geq 0$. Then, it is sufficient to have $m\geq 48\theta \kappa/[\gamma (1-\lambda_2)]$ and $\gamma\leq \min\crki{\alpha^2(1-\lambda_2)/[15876C], 1/2}$ for \eqref{eq:etak_s1} hold.
	
	Then,
	\begin{equation}
		\label{eq:cLv_scvx_s1}
		\begin{aligned}
			&\E\cL_{k + 1}^v \leq \prt{1 - \frac{\eta_k \mu}{2}}\E\brk{f(\bar{x}_k) - f^*}+ \brk{\frac{7 + \tlambda_2}{8} \right.\\
			&\left. +\frac{32\gamma C[31\gamma + 18L^2(\eta_k - \eta_{k + 1})^2]}{\alpha^2} }\frac{8\eta_k L^2}{n(1-\tlambda_2)}\E\brk{\norm{\ch_k}^2}\\
			& + \brk{\frac{\alpha}{4} + \prt{1 - \frac{\alpha}{3}} } \frac{56\gamma C\eta_k L^2}{n\alpha (1-\tlambda_2)}\E\brk{\norm{\y_k - H_k}^2}\\
			& + \frac{200(\eta_k - \eta_{k + 1})^2\eta_k L^2\sumn\norm{\nabla f_i(x^*)}^2}{n(1-\tlambda_2)^3}\\
			& + \frac{L\eta_k^2\sigma^2}{2n} + \frac{42L^2\eta_k^3\sigma^2}{1-\tlambda_2} .
		\end{aligned}
	\end{equation}

	Letting $m\geq 48\theta \kappa/[\gamma (1-\lambda_2)]$ and $\gamma\leq \min\crki{\alpha^2(1-\lambda_2)/(15876C), 1/2, \alpha / 2}$, it can be verified that 
	\begin{equation}
		\label{eq:etak_s2}
		\begin{aligned}
			&\frac{32\gamma C[31\gamma + 18L^2(\eta_k - \eta_{k + 1})^2]}{\alpha^2} \leq \frac{1-\tlambda_2}{8} - \frac{\eta_k\mu}{2},\\
			&1 - \frac{\alpha }{12} \leq 1-\frac{\eta_k \mu}{2}.
		\end{aligned}
	\end{equation}

	Substituting \eqref{eq:etak_s2} into \eqref{eq:cLv_scvx_s1} leads to the desired recursion for $\E\cL_k^v$.
	
}

\subsection{Proof of Lemma \ref{lem:lya}}
\label{app:lem_lya}
	\textbf{Step 1: Two basic results for bounding specific sequences.} To begin with, we introduce two basic results that help us to unroll two specific recursions. 
	\begin{lemma}
        \label{lem:seq_ab}
        For a sequence of positive numbers $\crk{a_k}$ and $\crk{b_k}$, let $k + m >  b_1 > p + 1$ and $m\geq \max\crki{\frac{2b_1}{\ln\frac{1}{1-c_1}},2b_1}$ if $c_2\neq 0$ and $m\geq 2b_1$ otherwise.
        If $c_1\in (0, 1), 0\leq c_2 < k + m$ and $a_{k + 1}\leq [1 - b_1/(k + m)]a_k + \sum_{i=2}^p \frac{b_i}{(k + m)^i} + (1 - c_1)^k\frac{c_2}{k + m}, $
        then we have 
        \begin{align*}
            a_{k}&\leq \prt{\frac{m}{k + m}}^{b_1}a_0 + \sum_{i=2}^p\frac{2b_i}{b_1 - i+1}\frac{1}{(k + m)^{i-1}} \\
			&\quad + \frac{4c_2 m^{b_1 - 1}}{(k + m)^{b_1}\ln(\frac{1}{1-c_1})}.
        \end{align*}

    \end{lemma}
	{
    \begin{proof}
        This is a direct application of \cite[Lemma 11]{pu2021distributed}. We have 
		\small
		\begin{align*}
			&a_k \leq \prod_{t = 0}^{k - 1}\prt{1 - \frac{b_1}{t+m}}a_0 + \sum_{t=0}^{k-1}\crk{\brk{\prod_{j = t + 1}^{k-1}\prt{1 - \frac{b_1}{j + m}}}\right.\\
			&\quad\left. \cdot \brk{\sum_{i=2}^p \frac{b_i}{(t + m)^i} + (1-c_1)^t\frac{c_2}{t + m}}}\\
			&\leq \prt{\frac{m}{k + m}}^{b_1}a_0\\
			&\quad + \sum_{t=0}^{k-1}\prt{\frac{t + m + 1}{k + m}}^{b_1}\brk{\sum_{i=2}^p \frac{b_i}{(t + m)^i} + \frac{(1 - c_1)^t c_2}{t + m}}\\
			&\leq \prt{\frac{m}{k + m}}^{b_1}a_0 + \sum_{i=2}^p\frac{2b_i}{b_1 - i+1}\frac{1}{(k + m)^{i-1}}\\
			&\quad - \frac{4c_2 m^{b_1 - 1}}{(k + m)^{b_1}\ln(1-c_1)}.
		\end{align*}\normalsize

		The last term holds because, for the chosen $m$, we have
		\small
		\begin{align*}
			&\frac{1}{\ln (1-c_1)}\mathrm{d}\prt{(1-c_1)^t (t + m)^{b_1 - 1}} = (1-c_1)^t(t + m)^{b_1 - 1}\mathrm{dt}\\
			&+ \frac{b_1 - 1}{\ln(1-c_1)} (1-c_1)^t (t + m)^{b_1-2}\mathrm{dt}\geq \frac{(1-c_1)^t (t + m)^{b_1 - 1}\mathrm{dt}}{2}.
		\end{align*}\normalsize

		Hence, 
		\begin{align*}
			&\sum_{t=0}^{k - 1}(1-c_1)^t(t+m)^{b_1 - 1}
			\leq \int_0^{+\infty}(1-c_1)^t(t+m)^{b_1 - 1}\mathrm{dt}\\
			&\leq \frac{2}{\ln(1-c_1)}\int_0^{+\infty}\mathrm{d}\prt{(1-c_1)^t (t + m)^{b_1 - 1}} = \frac{2m^{b_1 - 1}}{\ln\frac{1}{1-c_1}}.
		\end{align*}
		
    \end{proof}
	}

	\begin{lemma}
		\label{lem:seq_abc}
		Suppose we have two sequences of positive numbers $\crk{a_k}$ and $\crk{b_k}$ satisfying $a_{k + 1}\leq (1-b_1) a_k + \sum_{i=2}^p \frac{b_i}{(k + m)^i},\ b_1<1,$
		and $m\geq 1/[{1-\prt{1 -b_1/2}^{1/p}}]$,
		then $a_k \leq (1-b_1)^{k}a_0 + \sum_{i=2}^p \frac{2b_i}{b_1(k + m)^i}$.
	\end{lemma}
	{
	\begin{proof}
		Denote $A_i(k) := \sum_{t=0}^{k - 1}(1-b_1)^{k -t- 1} b_i/[(t + m)^i].$  We have $a_k \leq (1-b_1)^{k}a_0 + \sum_{i=2}^p A_i(k)$
		and $A_i(k + 1) = (1-b_1)A_i(k) + b_i/[(k + m)^i]$.
		Hence, we can show by induction that 
		\begin{align*}
			A_i(k) &\leq \frac{b_i}{\brk{\prt{1 - \frac{1}{m + 1}}^i - (1 - b_1)}(k + m)^i}\leq \frac{2b_i}{b_1(k + m)^i},
		\end{align*}
		which gives the desired result.
	\end{proof}
	}

	\textbf{Step 2: Bounding $\E\cL_k^{{ v}}$.}
	Recall that we have $(\eta_k - \eta_{k + 1})^2 \leq \theta^2/[\mu^2(k + m)^4]$ for the decreasing stepsize policy $\eta_k$.
	Then \eqref{eq:cLv_scvx} becomes 
	\begin{equation}
		\label{eq:l1_p}
		\begin{aligned}
			\E\cL_{k + 1}^{ v} &\leq \prt{1 - {\frac{\theta}{2(k + m)}}} \E\cL_k^{ v} + \frac{p_1}{(k + m)^2} + \frac{p_2}{(k + m)^3}\\
			&\quad { + \frac{p_3}{(k + m)^5}}.
		\end{aligned}
	\end{equation}

	Applying Lemma \ref{lem:seq_ab} to \eqref{eq:l1_p} by setting $c_2 = 0$, we obtain the recursion of $\E\cL_k^{ v}$ for $\theta > { 12}$ and $m \geq { \theta }$.
	{
	\begin{equation}
		\label{eq:l1_kpm}
		\begin{aligned}
			&\E\cL_k^v \leq \frac{4 p_1}{(\theta - 3)(k + m)} + \frac{4 p_2}{(\theta-5)(k + m)^2}\\
			&\quad  + \frac{4 p_3}{(\theta - 9)(k + m)^4}+ \prt{\frac{m}{k + m}}^{\frac{\theta}{2}}\E\cL_0^v.
		\end{aligned}
	\end{equation}
	}

	\textbf{Step 3: Bounding $\E\cL_k^b$}. { 
	Similar to the derivations in Lemma \ref{lem:lya_ncvx}, substituting the upper bound \eqref{eq:l1_kpm} into the recursions \eqref{eq:scvx_cons} and \eqref{eq:scvx_yh} and invoking \eqref{eq:lya2_cd} lead to}
	{
	\begin{equation}
		\label{eq:cL2}
		\begin{aligned}
			&\E\cL^b_{k + 1} \leq \brk{2\gamma C + \prt{1-\frac{\alpha }{3}}b_1}\E\brk{\norm{\y_k - H_k}^2}\\
			& + \brk{\frac{3 + \tlambda_2}{4}+ \frac{4 b_1[31\gamma  + 18L^2(\eta_k - \eta_{k + 1})^2]}{\alpha } }\E\brk{\norm{\ch_k}^2}\\
			&+ \frac{12n L\theta^2}{\mu^2(k + m)^2}\brk{\frac{2L(2+3\theta^2\kappa)}{\mu(1-\tlambda_2)^2(k + m)^2} + 8b_1 }\E\cL_k^v\\
			& + \prt{\frac{2}{(1-\tlambda_2)^2} + \frac{3b_1}{\alpha }}\frac{12\theta^2 \sumn\norm{\nabla f_i(x^*)}^2}{\mu^2(k + m)^4}\\
			& + \frac{\theta^2\prt{4 + 6b_1}n\sigma^2}{\mu^2(k + m)^2}.
		\end{aligned}
	\end{equation}
	}

	We next determine $b_1$ such that the following inequalities hold
	\begin{subequations}
		\label{eq:l2_0}
		\begin{align}
			& {\frac{4 b_1[31\gamma  + 18L^2(\eta_k - \eta_{k + 1})^2]}{\alpha }}\leq \frac{1-\tlambda_2}{12},\label{eq:l2_1}\\
			& \prt{1-\frac{\alpha }{3}}b_1 + 2\gamma C \leq \frac{5 + \tlambda_2}{6}b_1,\label{eq:l2_2}
		\end{align}
	\end{subequations}

	{ 
	Note that \eqref{eq:l2_2} is equivalent to $2\gamma C\leq \prti{\alpha / 3 - (1-\tlambda_2)/6}b_1$ and recall $\gamma \leq \alpha / [2(1-\lambda_2)]$, it is then sufficient to have $b_1 = 8\gamma C/\alpha$. Substituting such a $b_1$ into \eqref{eq:l2_1} yields 
	\begin{align*}
		\frac{32\gamma C[31\gamma + 18L^2(\eta_k - \eta_{k + 1})^2]}{\alpha^2}\leq \frac{1-\tlambda_2}{12},
	\end{align*}
	which is guaranteed according to \eqref{eq:etak_s2}. 
	Noting the conditions on $\gamma$ and $m$ and invoking \eqref{eq:l1_kpm}, we have from \eqref{eq:cL2} that 
	\begin{equation}
		\label{eq:cL2_s2}
		\begin{aligned}
			&\E\cL^b_{k + 1} \leq \frac{5+ \tlambda_2}{6}\E\cL_k^b + \frac{6\theta^2n\sigma^2}{\mu^2(k + m)^2} + \frac{p_4p_1}{(k + m)^3}\\
			& +  \frac{p_4p_2}{(k + m)^4}+ \frac{36\theta^2 \sumn\norm{\nabla f_i(x^*)}^2}{\mu^2(1-\tlambda_2)^2(k + m)^4}\\
			& + \frac{p_4 p_3}{(k + m)^5} + \frac{p_4 m^{\frac{\theta}{2}} \E\cL_0^v}{(k + m)^{\theta / 2 + 2}}.
		\end{aligned}
	\end{equation}
	}

	Letting $m\geq (16\theta )/(1-\tlambda_2)$ and applying Lemma \ref{lem:seq_abc} to \eqref{eq:cL2_s2} yield the desired upper bound for $\E\cL_k^b$ in \eqref{eq:cons1}.
	According to \eqref{eq:lya2_cd}, we have $\E\brki{\normi{{ \ch_k}}^2} \leq \E\cL_k^b$.

\bibliographystyle{IEEEtran}
\bibliography{references_all}

\begin{IEEEbiography}[{\includegraphics[width=1in,height=1.25in,clip,keepaspectratio]{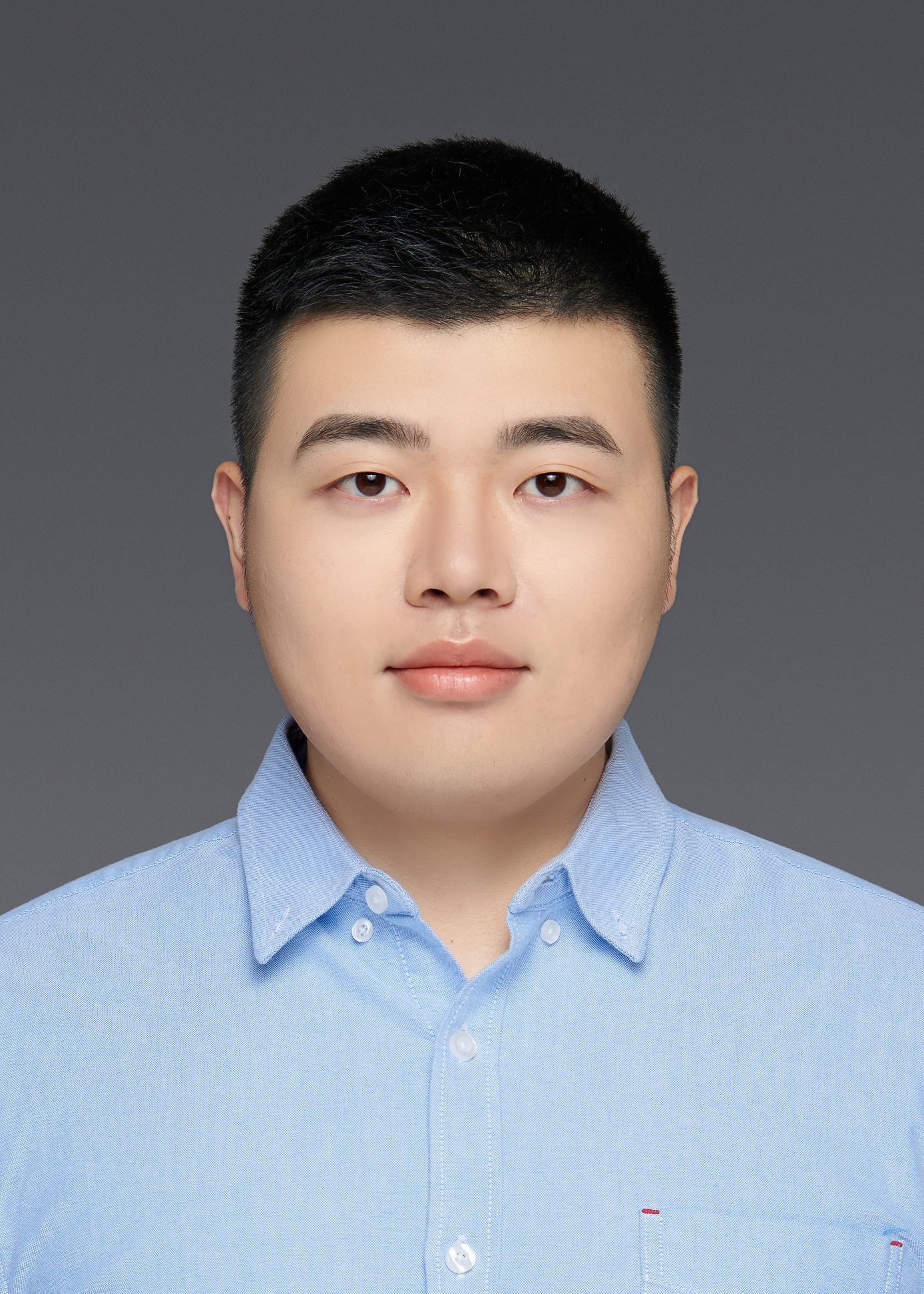}}]
	{Kun Huang} is currently a Ph.D. candidate in data science at the School of Data Science, The Chinese University of Hong Kong, Shenzhen, China. He obtained a B.S. degree in Applied Mathematics from Tongji University in 2018, and an M.S. degree in Statistics from the University of Connecticut in 2020. His research interests primarily lie in the fields of distributed optimization and machine learning.
\end{IEEEbiography}

\begin{IEEEbiography}[{\includegraphics[width=1in,height=1.25in,clip,keepaspectratio]{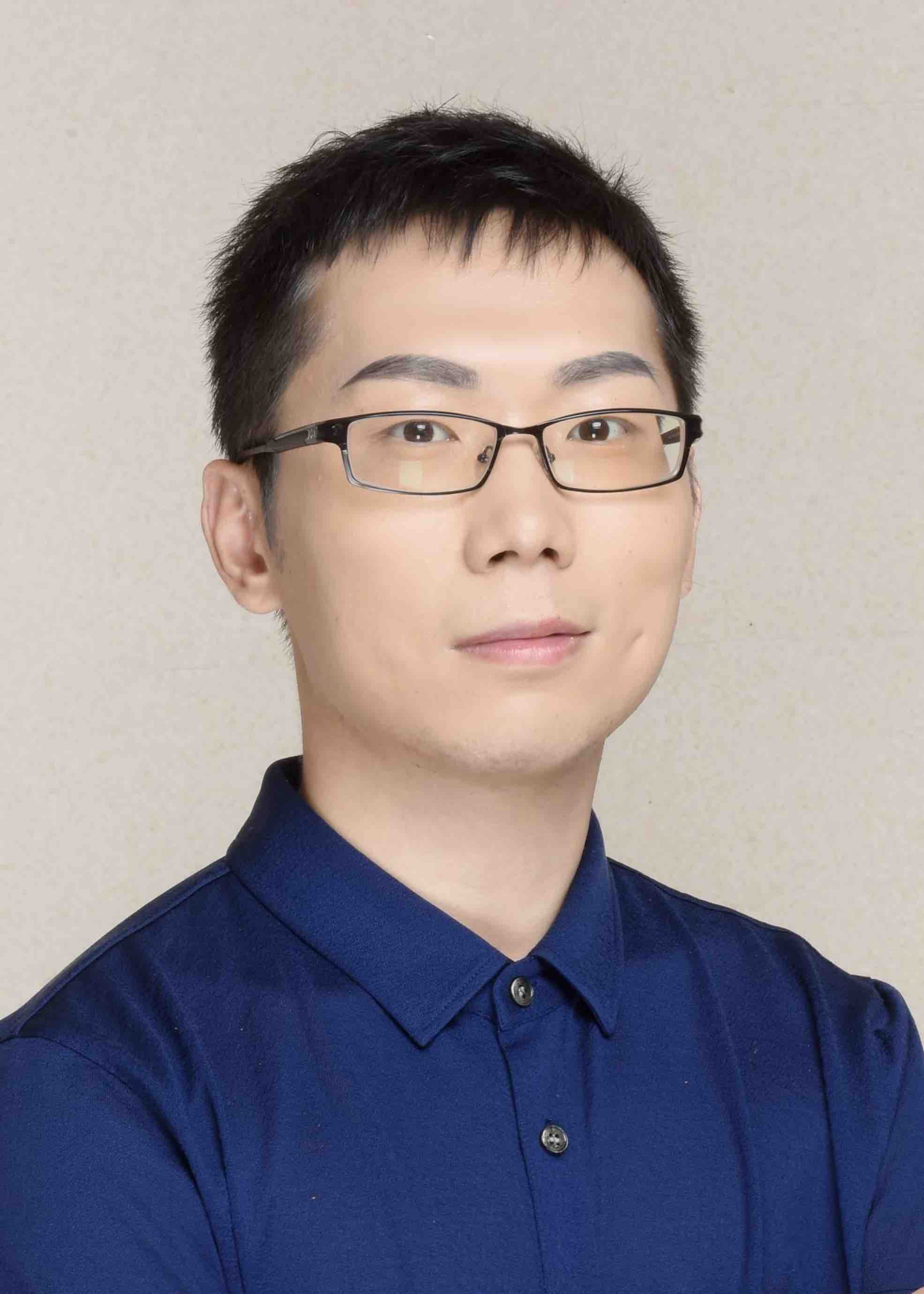}}]
	{Shi Pu} is currently an assistant professor in the School of Data Science, The Chinese University of Hong Kong, Shenzhen, China. He is also affiliated with Shenzhen Research Institute of Big Data. He received a B.S. Degree in engineering mechanics from Peking University, in 2012, and a Ph.D. Degree in Systems Engineering from the University of Virginia, in 2016. He was a postdoctoral associate at the University of Florida, Arizona State University and Boston University, respectively from 2016 to 2019. His research
            interests include distributed optimization, network science, machine learning, and game theory.
\end{IEEEbiography}

\end{document}